\documentclass[a4paper,reqno, 12pt]{amsart}

\usepackage{mystyle}
\usepackage{fullpage}

\usepackage{appendix}

\newcommand{\bari}{\overline{\imath}}
\newcommand{\barj}{\overline{\jmath}}
\newcommand{\On}{{\operatorname{O}(N)}}
\newcommand{\OGr}{{\operatorname{OGr}}}
\newcommand{\CHom}{\mathcal{H}om}

\newcommand{\Gortz}{{G\"ortz}\xspace}

\allowdisplaybreaks

\begin{document}
\title{On the flatness of spin local models for split even orthogonal groups}
\author{Jie Yang}
\address{Yau Mathematical Sciences Center, Tsinghua University, Haidian District, Beijing 100084, China}
\email{jie-yang@mail.tsinghua.edu.cn}

\begin{abstract}
    Let $F$ be a complete discretely valued field with ring of integers $\mathcal{O}$ and residue field of characteristic $p>2$. Let $G=\operatorname{GO}_{2n}$ denote the split orthogonal similitude group over $F$. For any parahoric level structure, we prove that the associated spin local model for $G$ is a flat $\mathcal{O}$-scheme with reduced special fiber. This confirms a conjecture of Pappas and Rapoport \cite[Conjecture 8.1]{pappas2009local} in the split case. As a corollary, we construct a flat (integral) moduli space of PEL-type D. 
\end{abstract}

\maketitle

\setcounter{tocdepth}{1}
\tableofcontents

\section{Introduction}\label{sec-introduction}
\subsection{}
Shimura varieties of PEL type are important objects in number theory. Many deep results about general Shimura varieties are first established in the PEL setting. A primary reason that they are relatively tractable is that these varieties can be realized as moduli spaces of abelian varieties with polarizations, endomorphisms and level structures.

A central arithmetic problem is to construct ``reasonable" integral models of Shimura varieties over $p$-adic integers. In \cite{rapoport1996period}, Rapoport and Zink constructed an integral version $\CA^\naive$ of the PEL-type moduli space with parahoric level structure at $p$ by extending the moduli problem from the reflex field to its ring of integers (localized at $p$). They showed that the corresponding Shimura variety is an open and closed subscheme of the generic fiber of $\CA^\naive$. Furthermore, $\CA^\naive$ is \etale locally isomorphic to the associated Rapoport--Zink local model $\RM^\naive$. Therefore, the flatness of $\CA^\naive$ is equivalent to the flatness of $\RM^\naive$.     When the underlying group involves only type A and C, and splits over an unramified extension of $\BQ_p$, \Gortz \cite{gortz2001flatness,gortz2003flatness} proved that $\RM^\naive$ is flat with reduced special fiber. However, Rapoport--Zink models are now called ``naive", because they are not always flat: Pappas \cite{pappas2000arithmetic}  first observed a failure of flatness in the case of ramified unitary groups, and Genestier later noted similar issues for split orthogonal groups.

One remedy for the non-flatness of $\RM^\naive$ is to refine the (naive) moduli problem by imposing additional linear algebraic conditions, thereby cutting out a closed subscheme $\RM\sset\RM^\naive$ that is flat with the same generic fiber as $\RM^\naive$. 
Unlike the cases for symplectic groups and unitary groups (see \cite{gortz2001flatness,gortz2003flatness, pappas2000arithmetic, pappas2009local, smithling2015moduli, yu2019moduli, luo}, etc), local models for orthogonal groups have been less systematically investigated in the literature (but see \cite{smithling2011topological, he2020good, Zachos23}) for related results).

We now introduce some notation used in our work.
Let $F$ be a complete discretely valued field with ring of integers $\CO$, uniformizer $\pi_0$, and residue field $k$ of characteristic\footnote{The case $p=2$ is substantially different and more difficult; we do not treat it here. } $p>2$. For an integer $n\geq 1$, set $V=F^{2n}$ and equip $V$ with a non-degenerate symmetric $F$-bilinear form $$\psi\colon V\times V\ra F.$$ Assume $\psi$ is split, i.e., $V$ admits an $F$-basis $(e_i)_{1\leq i\leq 2n}$ such that $\psi(e_i,e_j)=\delta_{i,2n+1-j}$. Denote by $G\coloneqq \GO(V,\psi)$ the (split) orthogonal similitude group associated with $(V,\psi)$. Let $G^\circ$ be the connected component of $G$ containing the identity element. Let $\CL$ be a self-dual periodic lattice chain of $V$ in the sense of \cite[Chapter 3]{rapoport1996period}. For each $\Lambda\in\CL$, we let $$\Lambda^\psi\coloneqq \cbra{x\in V\ |\ \psi(x,\Lambda)\sset \CO } $$ denote the dual lattice of $\Lambda$ with respect to $\psi$.

\begin{defn}[{cf. \cite{rapoport1996period}}] \label{defn-naive}
    The naive local model $\RM^\naive_\CL$ is the functor \[\Sch_{/\CO}^\op\ra \Sets \] sending a scheme $S$ over $\CO$ to the set of $\CO_S$-modules $(\CF_\Lambda)_{\Lambda\in \CL}$ such that 
    \begin{itemize}
        \item [LM1.] for any $\Lambda\in \CL$, Zariski locally on $S$, $\CF_\Lambda$ is a direct summand of $\Lambda_S\coloneqq \Lambda\otimes_\CO\CO_S$ of rank $n$;
        \item[LM2.] for any $\Lambda\in \CL$, the perfect pairing \[\psi\otimes 1\colon \Lambda_S\times \Lambda^\psi_S \ra \CO_S \] induced by $\psi$ satisfies $(\psi\otimes 1)(\CF_\Lambda, \CF_{\Lambda^\psi})=0$;
        \item[LM3.] for any inclusion $\Lambda\sset \Lambda'$ in $\CL$, the natural map $\Lambda_S\ra \Lambda'_S$ induced by $\Lambda\hookrightarrow \Lambda'$ sends $\CF_\Lambda$ to $\CF_{\Lambda'}$; and the isomorphism $\Lambda_S\simto (\pi_0\Lambda)_S$ induced by $\Lambda\overset{\pi_0}{\ra}\pi_0\Lambda$ identifies $\CF_\Lambda$ with $\CF_{\pi_0\Lambda}$.
    \end{itemize}
\end{defn}
The functor $\RM^\naive_\CL$ is representable by a projective scheme over $\CO$, which we will also denote by $\RM^\naive_\CL$. We have \begin{equation*}
    \RM^\naive_\CL\otimes F \simeq \OGr(n,V), 
\end{equation*} where $\OGr(n,V)$ denotes the orthogonal Grassmannian of (maximal) totally isotropic $n$-dimensional subspaces of $V$. In particular, the generic fiber of $\RM^\naive_\CL$
has dimension $n(n-1)/2$, and consists of two connected components $\OGr(n,V)_\pm$, which are isomorphic to each other. We can identify each connected component $\OGr(n,V)_\pm$ with the flag variety $G^\circ/P_{\mu_\pm}$, where $P_{\mu_\pm}$ denotes the parabolic subgroup associated to a minuscule cocharacter $\mu_\pm$ given by \begin{flalign}
    \mu_+\coloneqq (1^{(n)},0^{(n)}) \text{\ and\ } \mu_-\coloneqq (1^{(n-1)},0,1,0^{(n-1)}). \label{cochar}
\end{flalign}

In \cite[\S 7]{pappas2009local}, Pappas and Rapoport defined an involution $a\colon \wedge^{n}_FV\ra \wedge^n_FV$ inducing a decomposition  \begin{flalign}
	\wedge^{n}_FV= W_+\oplus W_-, \label{Wdecomp}
\end{flalign} where $W_\pm$ denotes the $\pm 1$-eigenspace for $a$. For an $\CO$-lattice $\Lambda$ in $V$, set \[(\wedge^{n}_\CO\Lambda)_\pm \coloneqq \wedge^{n}_\CO\Lambda\cap W_\pm.  \]
Then we define a refinement of $\RM^\naive_\CL$ as follows, see \cite[\S 2.3]{smithling2011topological}.


\begin{defn}\label{defn-spin}
    The spin local model $\RM^\pm_\CL$ is the functor \[\Sch_{/\CO}^\op\ra \Sets \] sending a scheme $S$ over $\CO$ to the set of $\CO_S$-modules $(\CF_\Lambda)_{\Lambda\in \CL}$, which satisfies LM1-3 and the condition 
    \begin{itemize}
        \item[LM4$\pm$.] for any $\Lambda\in\CL$, Zariski locally on $S$,  the line $\bigwedge_{\CO_{S}}^{n} \mathcal{F}_{\Lambda}$ is contained in
    \[
    \operatorname{Im}\left[ \left( \wedge_{\CO}^{n} \Lambda \right)_{\pm} \otimes_{\CO} \CO_{S} \longrightarrow \wedge_{\CO_{S}}^{n} \Lambda_{S} \right].
    \]
    \end{itemize}
\end{defn}

The functor $\RM^\pm_\CL$ is representable by a closed subscheme of $\RM^\naive_\CL$ over $\CO$, which we will also denote by $\RM^\pm_\CL$. Moreover, we have (cf. \cite[8.2.1]{pappas2009local}) \[\RM^\pm_\CL\otimes F\simeq \OGr(n,V)_\pm. \]

\begin{defn}\label{Mlocdefin}
    Denote by $\RM^{\pm\loc}_\CL$ the schematic closure of $\RM^\pm_\CL\otimes F$ in $\RM^\pm_\CL$.
\end{defn}

\begin{conj}[{cf. \cite[Conjecture 8.1]{pappas2009local}}] \label{introconj-pr}
    The spin local model $\RM^\pm_\CL$ is $\CO$-flat. Equivalently, we have $\RM^\pm_\CL=\RM^{\pm\loc}_\CL$.
\end{conj}

\subsection{Main results}
In the present paper, we prove the Pappas--Rapoport flatness conjecture for spin local models attached to split even orthogonal groups, with arbitrary parahoric level structure. 

\begin{thm}[{\S \ref{spinflat}}] \label{mainthm}
	Let $\CL$ be any periodic self-dual lattice chain. 
	Then the spin local model $\RM^\pm_\CL$ is $\CO$-flat, i.e., Conjecture \ref{introconj-pr} is true.
\end{thm}


\begin{remark}
	Note that Conjecture \ref{introconj-pr} can be verified by hand for $n\leq 3$. Following \cite{yangtopflat} (in order to treat the Bruhat--Tits-theoretic aspects of $\GO_{2n}$ uniformly), we assume $n\geq 4$ throughout the paper. 
\end{remark}

Denote by \begin{flalign}
	\sG\coloneqq \sG_{\CL}  \label{Gpara}
\end{flalign}  
the (affine smooth) group scheme of similitude automorphisms of $\CL$.  The neutral component $\sG^\circ$ is a parahoric group scheme for $G^\circ$ in the sense of Bruhat--Tits. 
Given $\sG^\circ$ and the cocharacter $\mu_\pm$ in \eqref{cochar}, one has the \dfn{canonical local model} $\RM^\loc_{\sG^\circ,\mu_\pm}$, the schematic local model representing the corresponding v-sheaf local model in the sense of Scholze--Weinstein \cite[\S 21.4]{scholze2020berkeley}. As a corollary of Theorem \ref{mainthm}, we have the following.

\begin{corollary}[{cf. Theorem \ref{yangthm}}]
    The canonical local model $\RM^\loc_{\sG^\circ,\mu_\pm}$ admits an explicit moduli interpretation given by $\RM^\pm_\CL$.
\end{corollary}

Theorem \ref{mainthm} has an application to PEL moduli spaces of type $D$. Let $(\bB,*,\bV,\pair{\ ,\ },\bh)$ be a PEL-datum of type $D$ as in \S \ref{sec-application}. It determines a reductive group $\bG$ over $\BQ$, whose base change to $\ol{\BQ}$ is isomorphic to $\GO_{2n,\ol{\BQ}}$. Suppose that $\bG_{\BQ_p}\simeq \GO_{2n,\BQ_p}$ is split. Then we can choose an open compact subgroup $\bK_p\sset \bG(\BQ_p)$ equal to $\sG_{\CL}(\BZ_p)$ for a self-dual chain $\CL$ (here $\CO=\BZ_p$ and $\pi_0=p$). Set $\bK=\bK_p\bK^p\sset \bG(\BA_f)$, where $\bK^p$ is a sufficiently small open compact subgroup of $\bG(\BA_f^p)$. 

We define a moduli functor $\CA_\bK^\pm$ of type $D$ parametrizing families of quadruples $$(A_\Lambda,\iota_\Lambda,\ol{\lambda}_\Lambda,\ol{\eta}_\Lambda)_{\Lambda\in \CL}$$ of abelian schemes with additional PEL data, subject to the spin condition (see \S \ref{sec-application} for more details). Then $\CA_\bK^\pm$ is representable by a quasi-projective scheme over $\BZ_p$, and the Shimura variety attached to the above PEL-datum with level $\bK$ is an open and closed subscheme of the generic fiber $\CA_{\bK}^\pm\otimes\BQ_p$.   Using the local model diagram, we conclude that $\CA^\pm_\bK$ is \etale locally isomorphic to the spin local model $\RM^\pm_\CL$.
\begin{corollary}[{Corollary \ref{coro-AKflat}}] \label{introcoro18}
    The (integral) moduli space $\CA^\pm_\bK$ of type $D$ is normal, Cohen--Macaulay, and flat over $\BZ_p$.
\end{corollary}

\subsection{}
Let us now outline the proof of Theorem \ref{mainthm}. Recall that in \cite{yangtopflat}, we prove that the spin local model is topologically flat (see Theorem \ref{yangthm}). It follows that, to prove Conjecture \ref{introconj-pr}, it suffices to show that the special fiber $\RM^\pm_{\CL,k}$ is reduced.

\begin{defn}[{Definition \ref{defn-stdchain}}]
    For $0\leq i\leq n$, set \begin{flalign*}
		\Lambda_i\coloneqq \CO\pair{\pi_0\inverse e_1,\ldots,\pi_0\inverse e_i,e_{i+1}, \ldots, e_{2n} }. 
	\end{flalign*}
	For any non-empty subset $I\sset [0,n]$ (notation as in \S \ref{subsec-notation-intro}), define the associated standard self-dual lattice chain via $$\Lambda_I\coloneqq \cbra{\Lambda_\ell}_{\ell\in 2n\BZ\pm I},$$ where $\Lambda_{-i}\coloneqq \Lambda_i^\psi$ and $\Lambda_\ell\coloneqq \pi_0^{-d}\Lambda_{\pm i}$, for $i\in I$ and $\ell=2nd\pm i, d\in\BZ$.
\end{defn}
By the classification (Theorem \ref{introthm-conjpara}) of $G^\circ(F)$-conjugation classes of parahoric subgroups, we may assume that $\CL=\Lambda_I$ for some $I\sset[0,n]$. Using a similar idea of \Gortz \cite[\S 4.5]{gortz2001flatness}, we reduce to the case $I=\cbra{i}$ is a singleton, which we call the \dfn{pseudo-maximal case}. We write $\RM_i^\naive$ (resp. $\RM^\pm_i$) for $\RM^\naive_{\cbra{i}}$ (resp. $\RM^\pm_{\cbra{i}})$. Let $\sG_i\coloneqq \sG_{\Lambda_{\cbra{i}}}$ denote the group scheme as in \eqref{Gpara}, with neutral component $\sG^\circ_i$. Then $\sG^\circ_i$ acts on $\RM^\pm_i$. 

\begin{thm}\label{intro-chart}
	There exists an open affine subscheme $\RU_i^\pm$ of $\RM^\pm_i$ such that the following holds.
	\begin{enumerate}
		\item The $\sG^\circ_i$-translates of $\RU^\pm_i$ cover $\RM^\pm_i$;
		\item The special fiber $\RU^\pm_{i,k}$ is reduced, and is isomorphic to
			   \begin{flalign*}
         \Spec \RR^\pm_{i,k} \times_k \BA_k^{(n-2i')(n+2i'-1)/2}, \text{\ for\ } \RR^\pm_{i,k}\coloneqq \frac{k[X]}{(XH_{2i'}X^t, X^tH_{2i'}X,\wedge^{i'+1}X)+\RI^\mp},
    \end{flalign*}
    where $X$ is a $2i'\times 2i'$ matrix with $i'\coloneqq \min\cbra{i,n-i}$, $H_{2i'}$ denotes the anti-diagonal unit $2i'\times 2i'$ matrix, ($\wedge^{i'+1}X$) denotes the ideal generated by all $(i'+1)$-minors of $X$, and $\RI^\mp$ denotes the ideal generated by \begin{flalign*}
        [S:T](X)\mp \sgn(\sigma_S)\sgn(\sigma_T)[S^\perp:T^\perp](X)
    \end{flalign*}
    for all subsets $S,T\sset[1,2i']$ (in increasing order) of cardinality $i'$. (See Proposition \ref{lem-wedgeN} for notation when $2i\leq n$.)
	\end{enumerate}
\end{thm}

It is clear that Theorem \ref{intro-chart} implies that $\RM^\pm_{i,k}$ is reduced. Let us explain the proof of Theorem \ref{intro-chart}; see Proposition \ref{lem-coverworst} and Proposition \ref{lem-Rikreduced} for details.
Part (1) follows from the description of the stratification (Proposition \ref{prop-stratificationMpm}) of the special fiber $\RM^\pm_{i,k}$: we take a worst point $P_0$ in the lowest stratum, and construct $\RU^\pm_i$ by intersecting $\RM^\pm_i$ with the standard affine chart of the Grassmannian around $P_0$ (see \S \ref{worstsec}). The proof of Part (2) rests on a detailed analysis of the defining equations of $\RU^\pm_{i,k}$, together with the following theorem, whose Part (2) implies that $\RR^\pm_{i,k}$ is reduced upon setting $N=2i'$. 
\begin{thm}[{\S \ref{subsec47} - \S \ref{subsec48}}] \label{intro-2}
      Let $N\geq 1$ be an integer. Set $m\coloneqq \lfloor N/2\rfloor$. \begin{enumerate}
      	\item For $0\leq \ell\leq m$, set \begin{flalign*}
      	\sR(\ell)\coloneqq \frac{\BZ[1/2][X]}{(XH_NX^t, X^tH_NX,\wedge^{\ell+1}X)},
      \end{flalign*}
      where $X$ is an $N\times N$ matrix, $H_{N}$ denotes the anti-diagonal unit $N\times N$ matrix, ($\wedge^{\ell+1}X$) denotes the ideal generated by all $(\ell+1)$-minors of $X$.
      
      Then $\sR(\ell)$ is reduced and flat over $\BZ[1/2]$. Moreover, $\sR(m)\otimes_{\BZ[1/2]}k$ is reduced; and for $0\leq \ell<m$, $\sR(\ell)\otimes_{\BZ[1/2]}k$ is irreducible, normal and Cohen--Macaulay. 
      \item Suppose $N=2m$ is even. Set \begin{flalign*}
      	\sR^\pm\coloneqq \frac{\BZ[1/2][X]}{(XH_NX^t, X^tH_NX,\wedge^{m+1}X)+\RI^\mp},
      \end{flalign*}
       where $\RI^\mp$ is defined similarly as in Theorem \ref{intro-chart} (2).
       Then the ring $\sR^\pm$ is reduced, Cohen--Macaulay, and flat over $\BZ[1/2]$. Moreover, $\sR^\pm\otimes_{\BZ[1/2]}k$ is reduced, Cohen--Macaulay and consists of two irreducible components. 
      \end{enumerate}  
\end{thm}

Our proof of Theorem \ref{intro-2}, the most technical part of the paper, involves the representation theory of orthogonal groups over $\BC$ and a variant of Hironaka's lemma (see Proposition \ref{prop-reducedZ}). 
The theorem itself is also of independent interest. To explain its proof, let us introduce some notation (see \S \ref{subsec-tableauxbasics} for details). 

\begin{defn}\label{introdefn-tableaux}
   Let $N\geq 1$ be an integer. Define the ordered set \begin{flalign*}
    \sI\coloneqq \begin{cases}
    	\cbra{\ov{1} <1<\ov{2}<2<\cdots<\ov{m}<m} &\text{if $N=2m$;} \\ \cbra{\ov{1} <1<\ov{2}<2<\cdots<\ov{m}<m<0} &\text{if $N=2m+1$.}    
    \end{cases} 
\end{flalign*}
   \begin{enumerate}
       \item A \dfn{partition} $\lambda$ of a positive integer $d$ into $r$ parts is given by writing $d$ as a sum $d=\lambda_1+\lambda_2+\cdots+\lambda_r$ of positive integers where $\lambda_1\geq \lambda_2\geq \cdots\geq \lambda_r$. We sometimes write $\lambda=(\lambda_1\geq \lambda_2\geq \cdots\geq \lambda_r)$.
       \item A \dfn{tableau} of shape $\lambda$ (or a \dfn{$\lambda$-tableau}) is a left justified array with $r$ rows, where the $i$-th row consists of $\lambda_i$ entries from the set $\sI$. For example, \[T=\begin{pmatrix}
           \ov{1} &\ov{2}\\ \ov{2} &2\\ 2
       \end{pmatrix} \]  is a tableau of shape $\lambda=(2\geq 2\geq 1)$. 
            
            For a tableau $T$, denote by $T^i$ the $i$-th column of $T$.
       \item The \dfn{conjugate} $\lambda'$ of a partition $\lambda$ is the partition whose parts are the column lengths of any (equivalently, every) tableau of shape $\lambda$. We write $\lambda'=(\lambda_1'\geq\lambda_2' \geq\cdots\geq \lambda_{r'}')$ for some integer $r'$.
       \item A tableau $T$ is \dfn{$GL(N)$-standard} if it has at most $n$ rows, and if the entries in each column are strictly increasing downward, and the entries in each row are non-decreasing from the left to the right.
       \item A tableau $T$ is \dfn{$O(N)$-standard} if it is $GL(N)$-standard and satisfies conditions in Definition \ref{defn-onstandard} (2).
       \item A \dfn{bitableau} $[S:T]$ of shape $\lambda$ consists of two tableaux $S$ and $T$ of the same shape $\lambda$. For example, \[[S:T]=\left[\begin{matrix}
            1 &1\\ 2 &\ov{2}\\ 3
        \end{matrix}\ :\ \begin{matrix}
            \ov{1} &1\\ \ov{2} &2\\ 3
        \end{matrix} \right] \] is a bitableau of shape $(2\geq 2\geq 1)$.  
       \item A bitableau $[S:T]$ is called $O(N)$-standard if both $S$ and $T$ are $O(N)$-standard.
       \item Let $R$ be a commutative ring. Let $X$ be an $N\times N$ matrix with columns and rows indexed by $\sI$. For a bitableau \[[S:T]=\left[\begin{array}{c}
            a_1  \\ a_2\\ \vdots\\ a_r
       \end{array}: \begin{array}{c}
            b_1  \\ b_2\\ \vdots\\ b_r
       \end{array} \right] \]
       with one column, we use $[S:T](X)$ to denote the determinant of the $r\times r$ submatrix of $X$ formed by the rows of indices $a_1,\ldots, a_r$ and columns of indices $b_1,\ldots,b_r$, i.e. $$[S:T](X)\coloneqq \sum_{\tau\in S_r}\sgn(\tau)x_{a_1\tau(b_1)}\cdots x_{a_r\tau(b_r)}\in R[X], $$ where $S_r$ denotes the permutation group on the set $\cbra{b_1,\ldots,b_r}$.
       For a general bitableau $[S:T]$ of shape $\lambda$, define \[[S:T](X)\coloneqq \prod_{i=1}^{\lambda_1}[S^i:T^i](X).\]
       We say $[S:T](X)$ is the \dfn{bideterminant} in $R[X]$ attached to $[S:T]$.

       For a quotient ring $R'$ of $R[X]$ and a bitableau $[S:T]$, the bideterminant $[S:T](X)$ maps to an element in $R'$, which is still denoted by $[S:T](X)$ (or simply $[S:T]$).
   \end{enumerate}
\end{defn}

The combinatorics of partitions and tableaux play an important role in the representation theory of classical groups. For example, finite dimensional irreducible representations of complex orthogonal groups $\On$ are  naturally indexed by partitions $\lambda$ with $\lambda_1'+\lambda_2'\leq N$. Moreover, King--Welsh \cite{king1992construction} constructed an explicit $\BC$-basis for each irreducible $\On$-representation in terms of $O(N)$-standard tableaux.

We first explain how to show the following part of Theorem \ref{intro-2}: for $0\leq \ell\leq m$, the ring \begin{flalign*}
    \sR(\ell)=  \frac{\BZ[1/2][X]}{(XHX^t,X^tHX,\wedge^{\ell+1}X)},\  H\coloneqq H_N,
\end{flalign*}
is reduced and flat over $\BZ[1/2]$. The same strategy, with some modifications, yields the reducedness and flatness of $\sR^\pm$.  
Denote \begin{flalign*}
    J\coloneqq J_N= \begin{pmatrix}
0 & 1 &        &        &        &        \\
1 & 0 &        &        &        &        \\
  &   & 0 & 1  &        &        \\
  &   & 1 & 0  &        &       \\
  &   &   &    & \ddots &       \\
  &   &   &    &  &0 & 1   \\
  &   &   &    &  &1 & 0   
\end{pmatrix},
\end{flalign*}
which is an $N\times N$ matrix consisting of $m$ diagonal blocks of $\begin{psmallmatrix}
    0 &1\\ 1 &0
\end{psmallmatrix}$. 
Then $\sR(\ell)$ is isomorphic to (Lemma \ref{twoRisom})  \begin{flalign*}
    \CR(\ell)\coloneqq {}_J\CR(\ell)=\frac{\BZ[1/2][X]}{(XJX^t,X^tJX,\wedge^{\ell+1}X)}.
\end{flalign*}

Denote by $\CR(\ell)_{\red}$ the associated reduced ring of $\CR(\ell)$. Adapting the straightening method of \cite{cliff2008basis}, we find generators for $\CR(\ell)$ (and hence for $\CR(\ell)_{\red}$) as a $\BZ[1/2]$-module via an explicit set of $O(N)$-standard bideterminants (see Definition \ref{defn-onstandard}). More precisely, there exists a surjection of $\BZ[1/2]$-modules \begin{flalign*}
    \varphi\colon \bigoplus_{[S:T]\in\CC(\ell)} \BZ[1/2]e_{[S,T]}\twoheadrightarrow \CR(\ell)\twoheadrightarrow\CR(\ell)_{\red},
\end{flalign*}
where $\tcbra{e_{[S,T]}}$ are basis elements indexed by \begin{flalign*}
    \CC(\ell)\coloneqq \cbra{O(N)\text{-standard bitableaux } [S:T]\ \vline\ \Centerstack{\text{the length of}\\ \text{the first column of $S$ is $\leq \ell$}} };
\end{flalign*}  see Corollary \ref{coro-span}. The image of each $e_{[S,T]}$ is the bideterminant $[S:T](X)$ in $\CR(\ell)$ or $\CR(\ell)_{\red}$. Thus, to show that $\CR(\ell)=\CR(\ell)_{\red}$ and that $\CR(\ell)$ is $\BZ[1/2]$-flat, it is enough to prove that $\ker(\varphi)\otimes_{\BZ[1/2]}\BC=0$. Namely, we need to show that the set $\varphi(\CC(\ell))$ is $\BC$-linearly independent in $\CR(\ell)_{\red,\BC}$. 

Denote by $\On$ the orthogonal group over $\BC$ attached to the symmetric matrix $J$. Set $O(N)\coloneqq \On(\BC)$. Note that the $\BC$-algebra $\CR(\ell)_{\BC}$ and $\CR(\ell)_{\red,\BC}$ carry a natural $O(N)$-bimodule structure via left/right multiplication on the matrix variable $X$.
Let $\wt{<}$ be a total order on the set of partitions. For each partition $\lambda$, we write $$A(\wt{\leq}\lambda)\sset \CR(\ell)_{\red,\BC} \ (\resp A(\wt{<}\lambda))$$ for the $\BC$-span of bideterminants in $\CC(\ell)$ of shape $\wt{\leq}\lambda$ (resp. $\wt{<}\lambda$). Then $A(\wt{\leq}\lambda)$ and $A(\wt{<}\lambda)$ are both $O(N)$-stable subspaces of $\CR(\ell)_{\red,\BC}$.

\begin{prop}[{Lemma \ref{lem-phiiso}, Corollary \ref{coro-Cbasis}}] \label{introprop}
    There exists a total order $\wt{<}$ such that, for every partition $\lambda$, there is an isomorphism of $O(N)$-bimodules 
    \begin{flalign*}
        \Phi(\lambda)\colon L(\lambda)\simto A(\wt{\leq} \lambda)/A(\wt{<}\lambda),
    \end{flalign*}
    where $L(\lambda)$ is an irreducible $O(N)$-bimodule, and the image of $\Phi(\lambda)$ has a $\BC$-basis consisting of bideterminants in $\CC(\ell)$ of shape $\lambda$.
\end{prop}

Now suppose that we have a linear relation in $\CR(\ell)_{\red,\BC}$: \begin{flalign*}
    c_1[S_1:T_1]+\cdots+ c_r[S_r:T_r] =0,
\end{flalign*}
where $c_i\in \BC$ and $[S_i:T_i]\in \CC(\ell)$. Then we choose a sufficiently large $\lambda$ with respect to the total order $\wt{<}$ in Proposition \ref{introprop} such that $[S_i:T_i]\in A(\wt{\leq}\lambda)$ for $1\leq i\leq r$. Since the images of $[S_i:T_i]$ in $A(\wt{\leq} \lambda)/A(\wt{<}\lambda)$ are linearly independent by Proposition \ref{introprop}, we obtain that $c_i=0$ for $1\leq i\leq r$. This shows that $\ker(\varphi_\BC)=0$, and hence, $\CR(\ell)$ is reduced and $\BZ[1/2]$-flat; cf. the proof of Corollary \ref{coro-Cbasis}.

As a byproduct of Proposition \ref{introprop}, we obtain a new basis of an irreducible $O(N)$-representation in terms of bideterminants. This is an orthogonal analogue of the result in \cite{de1979symplectic} for symplectic groups, and is of independent interest. 

To complete the proof of Theorem \ref{intro-2}, we proceed in the following steps. 
\begin{enumerate}
	\item We prove that the ring $\CR(\ell)$ (resp. $\CR^\pm$) is generically smooth of relative dimension $\ell(2N-2\ell-1)$ (resp. $m(2m-1)$) over $\BZ[1/2]$; see Proposition \ref{prop-Zlreduced} and Proposition \ref{prop-zjgensm}.
	\item We prove that $\CR(\ell)\otimes k$ for $0\leq\ell<m$ is irreducible and that  its reduction is a normal and Cohen--Macaulay domain. We also show that $\CR^\pm\otimes k$ consists of two irreducible components whose intersection is irreducible of codimension one. Moreover, these two irreducible components and their intersection are Cohen--Macaulay; see Proposition \ref{prop-zsred}. The Cohen--Macaulayness is deduced by relating these schemes to affine Schubert varieties. 
	\item We conclude that $\CR(\ell)\otimes k$ and $\CR^\pm\otimes k$ are reduced by applying a variant of Hironaka's lemma; see Proposition \ref{prop-reducedZ} and Corollary \ref{prop-Zk+red}. Some ideas in this step are inspired by \cite{yzz26}.
\end{enumerate}

\subsection{Organization}
We now give an overview of the paper. 

In \S \ref{sec-topoflat}, we review results in \cite{yangtopflat} that are used in this paper. The topological flatness (Theorem \ref{yangthm} and Corollary \ref{coro-ifreduced}) of the spin local models $\RM^\pm_I$ allows us to reduce Conjecture \ref{introconj-pr} to proving that the special fiber $\RM^\pm_{I,k}$ is reduced.  

In \S \ref{sec-coordrings}, we focus on the case $I=\cbra{i}$ is a singleton. By results in \S \ref{sec-topoflat}, we obtain an explicit stratification of $\RM^\pm_{i,k}$ (Proposition \ref{prop-stratificationMpm}), compatible with the decomposition into affine Schubert cells. Let $P_0$ (see \eqref{P0}) denote a closed point in the unique closed stratum of $\RM^\pm_{i,k}$. We call $P_0$ a \dfn{worst point} of $\RM^\pm_{i,k}$. We then construct an affine open subscheme $\RU^\pm_i\sset \RM^\pm_i$ containing the worst point, and show that $\RU^\pm_{i,k}$ is a closed subscheme of a product of $\Spec \RR^\pm_{i,k}$ and an affine space, where $\RR^\pm_{i,k}$ is an explicitly matrix-defined ring. 

In \S \ref{sec-flatness}, we prove our main technical result (Theorem \ref{intro-2}). Using straightening methods in \cite{cliff2008basis} and the representation theory of complex orthogonal groups in \cite{king1992construction}, we prove a class of rings defined by matrix equations (including $\RR^\pm_{i,k}$) is reduced. The proof relies on a variant of Hironaka's lemma, which we prove in Proposition \ref{prop-reducedZ}. As an application, we deduce that $\RU^\pm_{i,k}$ (and hence $\RM^\pm_{i,k}$) is reduced. Along the way, we also obtain an explicit $k$-basis of $\RR^\pm_{i,k}$ in terms of $O(N)$-standard bideterminants. 

In \S \ref{spinflat}, we prove the flatness result Theorem \ref{mainthm}. We first verify that the closed immersion $\RU^\pm_{i,k}\hookrightarrow \Spec \RR^\pm_{i,k}$ is in fact an isomorphism, using that $\RR^\pm_{i,k}$ is reduced. By Theorem \ref{intro-2}, we conclude Theorem \ref{mainthm} in the pseudo-maximal case $I=\cbra{i}$. By \Gortz's argument in \cite[\S 4.5]{gortz2001flatness}, this implies Theorem \ref{mainthm} for the general parahoric case.      

Finally in \S \ref{sec-application}, we apply Theorem \ref{mainthm} to moduli spaces of type $D$. Starting from the construction of Rapoport--Zink in \cite{rapoport1996period}, we impose the spin condition to modify the moduli problem. Using the local model diagram, we prove that the modified moduli space is flat.

 \subsection*{Acknowledgements}
It is a pleasure to thank M. Rapoport and Y. Luo for sharing an unpublished note of B. Smithling. Some of the ideas developed here are due to Smithling. I am also grateful to Y. Luo, G. Pappas, M. Rapoport, I. Zachos and Z. Zhao for their comments and suggestions, and to X. Shen for pointing out the reference \cite{yu2021reduction}. This work is partially supported by the Shuimu Tsinghua Scholar Program of Tsinghua University.

\subsection{Notation} \label{subsec-notation-intro}
Throughout the paper, $(F,\CO,\pi_0, k)$ denotes a complete discretely valued field with ring of integers $\CO$ , uniformizer $\pi_0$, and residue field $k$ of characteristic $p>2$. Denote by $\ol{k}$ an algebraic closure of $k$. 

For integers $n_2\geq n_1$, denote $[n_1,n_2]\coloneqq \cbra{n_1,\ldots,n_2}$. For $i\in [1,2n]$, define $i^*\coloneqq 2n+1-i$.  For a subset $E\sset [1,2n]$ of cardinality $n$, set $E^*\coloneqq \cbra{i^*\ |\ i\in E}$ and $E^\perp\coloneqq (E^*)^c$ (the complement of $E^*$ in $[1,2n]$). Denote by $\Sigma E$ the sum of the entries in $E$. 

The expression $(a^{(r)},b^{(s)},\ldots)$ denotes the tuple with $a$ repeated $r$ times, followed by $b$ repeated $s$ times, and so on. 

For any real number $x$, denote by $\lfloor x\rfloor$ (resp. $\lceil x\rceil$) for the greatest (resp. smallest) integer $\leq x$ (resp. $\geq x$).

For a ring $R$ and a matrix $X$, we write $R[X]$ for the polynomial ring over $R$ whose variables are entries of $X$.

For a scheme $\CX$ over $\CO$, we let $\CX_k$ or $\CX\otimes k$ denote the special fiber $\CX\otimes_\CO k$, and $\CX_F$ or $\CX\otimes F$ denote the generic fiber $\CX\otimes_\CO F$.

\section{Preliminaries}\label{sec-topoflat}
In this section, we review the main results proven in \cite{yangtopflat}, which reduce the proof of Conjecture \ref{introconj-pr} to showing that the special fiber $\RM^\pm_{\CL,k}$ is reduced. We follow the notation in \S \ref{subsec-notation-intro}.


\begin{defn}\label{defn-stdchain}
    For $0\leq i\leq n$, set \begin{flalign*}
		\Lambda_i\coloneqq \CO\pair{\pi_0\inverse e_1,\ldots,\pi_0\inverse e_i,e_{i+1}, \ldots, e_{2n} }. 
	\end{flalign*} 
	For any non-empty subset $I\sset [0,n]$, define the associated standard self-dual lattice chain via $$\Lambda_I\coloneqq \cbra{\Lambda_\ell}_{\ell\in 2n\BZ\pm I},$$ where $\Lambda_{-i}\coloneqq \Lambda_i^\psi$ and $\Lambda_\ell\coloneqq \pi_0^{-d}\Lambda_{\pm i}$, for $i\in I$ and $\ell=2nd\pm i, d\in\BZ$.
\end{defn}
	
\begin{thm}[{\cite[Theorem 1.2]{yangtopflat}}] \label{introthm-conjpara}
	Let $I\sset [0,n]$ be non-empty. 
	Then the subgroup \begin{flalign*}
		P_I^\circ\coloneqq \cbra{g\in G(F)\ |\ g\Lambda_i=\Lambda_i, i\in I} \cap\ker\kappa,
	\end{flalign*}
	where $\kappa\colon G^\circ(F)\twoheadrightarrow \pi_1(G)$ denotes the Kottwitz map, is a parahoric subgroup of $G(F)$. 
	
	Furthermore, any parahoric subgroup of $G(F)$ is $G^\circ(F)$-conjugate to a subgroup $P_I^\circ$ for some (not necessarily unique) $I\sset [0,n]$. In particular, any self-dual lattice chain $\CL$ is $G^\circ(F)$-conjugate to the standard self-dual lattice chain $\Lambda_I$ for some $I\sset [0,n]$. 
\end{thm}
If $\CL=g\Lambda_I$ for some $g\in G^\circ(F)$, then we have $\RM^\naive_\CL\simeq \RM^\pm_{\Lambda_I}$ and $\RM^\pm_\CL\simeq \RM^\pm_{\Lambda_I}$ (see \cite[Remark 3.3]{yangtopflat}). 
If furthermore $I=\cbra{i}\sset [0,n]$, then we say $\CL$ is \dfn{pseudo-maximal}. Set \begin{flalign*}
	\RM_I^\naive\coloneqq \RM^\naive_{\Lambda_I} \text{\ and\ } \RM^\pm_I\coloneqq \RM^\pm_{\Lambda_I}.
\end{flalign*} We also write $\RM^\naive_i$ (resp. $\RM^\pm_i$) for $\RM_{\cbra{i}}^\naive$ (resp. $\RM^\pm_{\cbra{i}}$).

For an $\CO$-algebra $R$, we will use $(\CF_\ell)_{\ell\in 2n\BZ\pm I}$, where $\CF_\ell\sset \Lambda_\ell\otimes_\CO R$, to denote an $R$-point of $\RM^\naive_I$ or $\RM^\pm_I$.

Conjecture \ref{introconj-pr} is then equivalent to the following.
\begin{conj}\label{conjPR1}
	For any non-empty subset $I\sset[0,n]$, the spin local model $\RM_I^\pm$ is flat over $\CO$. 
\end{conj}

Let $\sG_I$ be the group scheme of similitude automorphisms of $\Lambda_I$. Then $\sG^\circ_I$ is the parahoric group scheme with $\sG_I^\circ(\CO)=P_I^\circ$.  
 Recall that given $\sG^\circ_I$ and the minuscule cocharacter $\mu_\pm$ in \eqref{cochar}, the associated \dfn{schematic local model} $\RM^\loc_{\sG^\circ,\mu_\pm}$ represents the corresponding v-sheaf local model in the sense of Scholze--Weinstein \cite[\S 21.4]{scholze2020berkeley}. Let $\RM^{\pm\loc}_I$ denote the schematic closure of $\RM^\pm_{I,F}$ in $\RM^\pm_I$. 
 
\begin{thm}\label{yangthm}
 	Let $I\sset [0,n]$ be non-empty. \begin{enumerate}
 		\item (\cite[Proposition 3.6]{yangtopflat}) The scheme $\RM^{\pm\loc}_I$ is isomorphic to $\RM^\loc_{\sG^\circ_I,\mu_\pm}$. In particular, $\RM^{\pm\loc}_I$ is $\CO$-flat of relative dimension $n(n-1)/2$, normal, and Cohen--Macaulay with reduced special fiber. 
 		\item (\cite[Theorem 1.3]{yangtopflat}) The spin local model $\RM^\pm_I$ is topologically flat over $\CO$, i.e., the generic fiber $\RM_{I,F}^\pm$ is dense in $\RM^\pm_I$ (with respect to the Zariski topology).
 	\end{enumerate} 
\end{thm}

\begin{corollary}\label{coro-ifreduced}
	If the special fiber $\RM^\pm_{I,k}$ is reduced, then Conjecture \ref{introconj-pr} is true.
\end{corollary}	
\begin{proof}
	By assumption and Theorem \ref{yangthm} (2), the closed immersion $\RM^{\pm\loc}_I\hookrightarrow \RM^\pm_I$ is an isomorphism. Hence, $\RM^\pm_I$ is flat over $\CO$.
\end{proof}

Recall that over the function field $k((t))$, we \cite{yangtopflat} defined\footnote{In \loccit, we used the notation $G$ and $\CP_I$.} the equal-characteristic analogues $G^\flat$ and $\sG_I^\flat$ corresponding to $G$ and $\sG_I$, respectively.  By \cite[Corollary 3.22]{yangtopflat}, we have a closed embedding \begin{flalign*}
	\RM_{I,k}^\pm\hookrightarrow LG^{\flat\circ}/L^+\sG_I^{\flat\circ},
\end{flalign*}
identifying $\RM^\pm_{I,k}$ with a union of Schubert varieties in the affine flag variety $LG^{\flat\circ}/L^+\sG_I^{\flat\circ}$.
By construction, for each $i\in I$, we have a natural $L^+\sG_I^{\flat\circ}$-equivariant commutative diagram \begin{flalign*}
	 \xymatrix{
	    \RM^\pm_{I,k}\ar@{^{(}->}[r]\ar[d] &LG^{\flat\circ}/L^+\sG_I^{\flat\circ}\ar[d]^{\rho_i}\\ \RM^\pm_{i,k}\ar@{^{(}->}[r] &LG^{\flat\circ}/L^+\sG_i^{\flat\circ},
	 }
\end{flalign*} 
and we have (cf. \cite[Lemma 4.15]{yangtopflat}) \begin{flalign}
	\RM^\pm_{I,k} = \bigcap_{i\in I}\rho_i\inverse(\RM^\pm_{i,k}) \label{MIintersec}
\end{flalign}
as a schematic intersection. Here, $\rho_i\inverse(\RM^\pm_{i,k})$ is the pullback of $\RM^\pm_{i,k}$ along $\rho_i$.

By construction, the naive local model $\RM^\naive_i$ is isomorphic to the projective $\CO$-scheme representing the functor \begin{flalign*}
	\RM^\naive_i\colon \Sch_{/\CO}^\op \ra \Sets
\end{flalign*}
sending an $\CO$-scheme $S$ to the set of pairs $(\CF,\CG)$ of $\CO_S$-modules such that \begin{enumerate}
	\item [LM1.] Zariski locally on $S$, the $\CO_S$-module $\CF_i$ (resp. $\CF_{-i}$) is a direct summand of $\Lambda_{i,S}$ (resp. $\Lambda_{-i,S}$) of rank $n$;
	\item [LM2.] the perfect pairing $$\psi\otimes 1\colon \Lambda_{i,S}\times\Lambda_{-i,S} \ra \CO_S$$ induced by $\psi$ satisfies $(\psi\otimes 1)(\CF_i,\CF_{-i})=0$;
	\item [LM3.] the natural maps \begin{flalign*}
		  \Lambda_{-i,S}\xrightarrow{\iota_1} \Lambda_{i,S}\xrightarrow{\iota_2} \pi_0\inverse\Lambda_{-i,S}
	\end{flalign*}
	induced by the inclusions $\Lambda_{-i}\sset \Lambda_i\sset \pi_0\inverse\Lambda_{-i}$ satisfy \begin{flalign*}
		\text{$\iota_1(\CF_{-i})\sset \CF_i$ and $\iota_2(\CF_i)\sset \pi_0\inverse\CF_{-i}$.}
	\end{flalign*} 
\end{enumerate}
\begin{remark}
	It follows from condition LM2 that a point $(\CF_i,\CF_{-i})$ in $\RM^\naive_i$ is uniquely determined by $\CF_i$ or $\CF_{-i}$.
\end{remark}

\begin{defn}[{cf. \cite[Definition 1.7]{yangtopflat}}]  \label{introdefn-Ml}
    Let $\iota_2\colon \Lambda_i\ra \pi_0\inverse\Lambda_{-i}$ denote the natural inclusion map (and its base change). 
    For an integer $\ell$, denote by $$S_i(\ell)\sset \RM^\naive_{i,k}$$ the closed subscheme such that for any $k$-algebra $R$, $(\CF_i,\CF_{-i})\in \RM^\naive_{i,k}(R)$ lies in $S_i(\ell)(R)$ if and only if $\wedge^{\ell+1}(\iota_2: \CF_i\ra \pi_0\inverse\CF_{-i})=0$. 
    
    Let $\RM^\pm_i(\ell)\sset \RM^\pm_{i,k}$ denote the locus where $\iota_2(\CF_i)$ has rank $\ell$. Then $\RM^\pm_i(\ell)$ is a (reduced) locally closed subscheme of $\RM^\pm_{i,k}$.
\end{defn}

\begin{prop} \label{prop-stratificationMpm}
	Suppose that $I=\cbra{i}\sset [0,n]$. 
	\begin{enumerate}
	    \item The underlying topological space of $\RM^\naive_i$ is the union of those of $\RM^+_i$ and $\RM^-_i$.
		\item If $i=0$ or $n$, then $\RM^\pm_i$ is isomorphic to a connected component of the orthogonal Grassmannian $\OGr(n,2n)$ over $\CO$. In particular, $\RM^\pm_i$ is irreducible and smooth of relative dimension $n(n-1)/2$.
		\item If $i\neq 0,n$, then there exists a stratification of the reduced special fiber \begin{flalign}
        (\RM^\pm_{i,k})_\red=\coprod_{\ell=\max\cbra{0,2i-n}}^i \RM^\pm_i(\ell), \label{stratification}
    \end{flalign}
    where the top stratum $\RM^\pm_i(i)$ decomposes into exactly two Schubert cells (of dimension $n(n-1)/2$), and each lower stratum $\RM^\pm_i(\ell)$ for $\max\cbra{0,2i-n}\leq \ell<i$ is a single Schubert cell. 
    
    Moreover, as topological subspaces of $\RM^\naive_{i,k}$, $S_i(i)$ is the union of $\RM^+_i$ and $\RM^-_i$; and for $\max\cbra{0,2i-n}\leq \ell<i$, $S_i(\ell)_\red$ is a Schubert variety in $\RM^\pm_{i,k}$. 
	\end{enumerate} 
\end{prop}
\begin{proof}
	See \cite[Remark 3.8, Corollary 4.11, Proposition 4.13]{yangtopflat}. 
\end{proof}

\section{Local coordinate rings of spin local models}\label{sec-coordrings}

In this section, we retain the notation of \S \ref{sec-introduction} and \S \ref{sec-topoflat}. We focus on the pseudo-maximal parahoric case $I=\cbra{i}\sset [0,n]$.  We will provide an explicit description of the coordinate ring of an affine neighborhood $\RU^\naive_i$ containing the worst point in $\RM^\naive_i$. We give an explicit matrix-defined ring $\RR^\pm_{i,k}$ such that $\RU^\pm_{i,k}=\RU^\naive_{i,k}\cap \RM^\pm_{i}$ is (isomorphic to) a closed subscheme of a product of $\Spec \RR^\pm_{i,k}$ and an affine space.

\subsection{The worst point of $\RM^\pm_i$}  \label{worstsec}

Fix an ordered $\CO$-basis $(f_j)_{j=1}^{2n}$ of $\Lambda_i$ \begin{flalign}
    \pi_0\inverse e_1,\ldots,\pi_0\inverse e_i,e_{i+1},\ldots,e_{2n},  \label{basis1}
\end{flalign}  
and an ordered $\CO$-basis $(g_j)_{j=1}^{2n}$ of $\Lambda_{-i}$ \begin{flalign}
    e_1,\ldots,e_{n-i},\pi_0 e_{n-i+1},\ldots,\pi_0 e_{2n}.  \label{basis2}
\end{flalign}   
By abuse of notation, we also use $f_j$ (resp. $g_j$) to denote the $R$-basis of the base change $\Lambda_i\otimes R$ (resp. $\Lambda_{-i}\otimes R)$ for an $\CO$-algebra $R$. 

Set \begin{flalign*}
	\CF_{i,0} &\coloneqq k\pair{f_{i+1},\ldots,f_{n+i}}\sset \Lambda_i\otimes k, \\ \CF_{-i,0} &\coloneqq k\pair{g_1,\ldots,g_{n-i},g_{2n-i+1},\ldots,g_{2n}}\sset\Lambda_{-i}\otimes k.
\end{flalign*}
By Definition \ref{introdefn-Ml} and Proposition \ref{prop-stratificationMpm},  the pair \begin{flalign}
	P_0\coloneqq (\CF_{i,0},\CF_{-i,0}) \label{P0}
\end{flalign}  defines a point in the lowest stratum $\RM^\pm_i(\ell_0)$ for $\ell_0\coloneqq \max\cbra{0,2i-n}$.
Since each stratum in the decomposition \eqref{stratification} is stable under the action of the parahoric group scheme $\sG^\circ_i$, $\RM^\pm(\ell_0)$ is the unique closed $\sG^\circ_i$-orbit. Then we obtain the following proposition.
\begin{prop}\label{lem-coverworst}
    The $\sG^\circ_i$-translates of any non-empty open affine neighborhood of $P_0$ in $\RM^\pm_i$ cover the whole scheme $\RM^\pm_i$. 
\end{prop} 
We will call $P_0$ the \dfn{worst point} of $\RM^\pm_i$. 

By construction, we can view $\RM^\naive_i$ as a closed subscheme of $\Gr(n,\Lambda_i)\times \Gr(n,\Lambda_{-i})$.
The standard open affine neighborhood $\RU^{\mathrm{Gr}}_i$ in $\Gr(n,\Lambda_i)\times\Gr(n,\Lambda_{-i})$ containing $P_0$ is the $\CO$-scheme of two $2n\times n$ matrices \[\begin{pmatrix}
    M\\ I_n\\ M'
\end{pmatrix} \text{\ and\ } \begin{pmatrix}
    0 & &I_{n-i}\\ &N  \\ I_i & &0
\end{pmatrix},  \]
where $M$ is of size $i\times n$, $M'$ is of size $(n-i)\times n$, and $N$ is of size $n\times n$. The point $P_0$ corresponds to the closed point defined by $M=M'=N=\pi_0=0$. Define \[\RU_i^\pm\text{\ and\ } \RU_i^\naive \] by intersecting $\RU_i^\mathrm{Gr}$ with $\RM^\pm_i$ and $\RM^\naive_i$, respectively. 

\begin{corollary}\label{coro-Ureduced}
	If $\RU^\pm_{i,k}$ is reduced, then $\RM^\pm_{i,k}$ is reduced.
\end{corollary}
\begin{proof}
	This follows easily from Proposition \ref{lem-coverworst}.
\end{proof}

\subsection{The coordinate rings of $\RU_i^\pm$ and $\RU_i^\naive$}
We will analyze the above affine open neighborhoods around $P_0$ in detail. By Proposition \ref{prop-stratificationMpm} (4), we may assume that $2i\leq n$. We divide the matrices $M$, $M'$, and $N$ into blocks:
\begin{flalign*}
    &\begin{array}{c@{}c@{}c@{}c}
         M'= &\left(\begin{array}{ccc}
              M_9 &M_6 &M_3  \\
              M_8 &M_5 &M_2 
         \end{array}\right) &\begin{array}{c}
              \raisebox{-0.5ex}{\mbox{\tiny $n-2i$}}\vspace{5pt}  \\
              \raisebox{0.5ex}{\mbox{\tiny $i$}}
         \end{array}  \\
         &\begin{array}{ccc}
            \raisebox{0.5ex}{\mbox{\tiny $n-2i$}} &\raisebox{0.5ex}{\mbox{\ \tiny $i$}} &\raisebox{0.5ex}{\mbox{\quad \tiny $i$}}  
         \end{array} 
    \end{array}, \qquad  
    \begin{array}{c@{}c@{}c@{}c}
         M= &\left(\begin{array}{ccc}
              M_7 &M_4 &M_1  
         \end{array}\right) &\begin{array}{c}
              \raisebox{0ex}{\mbox{\tiny $n-i$}}
         \end{array}  \\
         &\begin{array}{ccc}
             \raisebox{0ex}{\mbox{\tiny $n-2i$}} &\raisebox{0ex}{\mbox{\ \tiny $i$}} &\raisebox{0ex}{\mbox{\quad \tiny $i$}}  
         \end{array} 
    \end{array}, \\
    &\hspace{3cm}
    \begin{array}{c@{}c@{}c@{}c}
          N= &\left(\begin{array}{ccc}
               N_1 &N_2 &N_3  \\
               N_4 &N_5 &N_6 \\
               N_7 &N_8 &N_9
          \end{array}\right) &\begin{array}{c}
               \raisebox{-0.5ex}{\mbox{\tiny $i$}}\vspace{5pt}  \\
               \raisebox{0.5ex}{\mbox{\tiny $i$}} \\
               \raisebox{0.5ex}{\mbox{\tiny $n-2i$}}
          \end{array}  \\
          &\begin{array}{ccc}
              \raisebox{0.5ex}{\mbox{\ \tiny $i$}} &\raisebox{0.5ex}{\mbox{\quad \tiny $i$}} &\raisebox{0.5ex}{\ \mbox{\tiny $n-2i$}}  
          \end{array} 
     \end{array}.
\end{flalign*}

\subsubsection{Condition LM2} \label{subsubsec-lm2}
With respect to the bases \eqref{basis1} and \eqref{basis2}, the matrix corresponding to the symmetric pairing $\psi$ is $H_{2n}$, the anti-diagonal unit matrix of size $2n\times 2n$. Condition LM2 translates to 
\begin{flalign*}
    \begin{pmatrix}
    M\\ I_n\\ M'
\end{pmatrix}^t H_{2n}  \begin{pmatrix}
    0 & &I_{n-i}\\ &N  \\ I_i & &0
\end{pmatrix} = 0.
\end{flalign*}
This is equivalent to 
\begin{flalign*}
    \begin{pmatrix}
        M_7^tH_{i}+H_{n-2i}N_7 &M_8^tH_{i}+H_{n-2i}N_8 &H_{n-2i}N_9+M_9^tH_{n-2i}\\ M_4^tH_{i}+H_{i}N_4 &M_5^tH_{i}+H_{i}N_5 &H_{i}N_6+M_6^tH_{n-2i}\\ 
        M_1^tH_{i}+H_{i}N_1 &M_2^tH_{i}+H_{i}N_2 &H_{i}N_3+M_3^tH_{n-2i}
    \end{pmatrix}=0.
\end{flalign*}
In particular, the block matrices from $M_1$ to $M_9$ are determined by block matrices from $N_1$ to $N_9$.

\subsubsection{Condition LM3} \label{subsubsec-LM3}
With respect to the bases \eqref{basis1} and \eqref{basis2}, the inclusion map $\Lambda_{-i}\hookrightarrow \Lambda_i$ and $\Lambda_i\hookrightarrow \pi_0\inverse\Lambda_{-i}$ are given by 
\begin{flalign*}
    \iota_1\coloneqq \begin{pmatrix}
        \pi_0 I_i &0 &0\\ 
        0 &I_{2n-2i} & 0\\ 
        0 &0 &\pi_0 I_{i}
    \end{pmatrix}
    \text{\ and\ }
    \iota_2 \coloneqq \begin{pmatrix}
        I_i &0 &0\\ 
        0 &\pi_0 I_{2n-2i} &0\\
        0 &0 &I_i
    \end{pmatrix}
\end{flalign*}
respectively. The condition that the natural map $\Lambda_{i}\otimes_\CO R\ra \pi_0\inverse\Lambda_{-i}\otimes_\CO R$ sends $\CF_{{i}}$ to $\pi_0\inverse\CF_{-i}$ is equivalent to that there exists an $n\times n$ matrix $A_2$ such that \begin{flalign*}
	\iota_2\begin{pmatrix}
    M\\ I_n\\ M'
\end{pmatrix} = \begin{pmatrix}
    0 & &I_{n-i}\\ &N  \\ I_i & &0
\end{pmatrix}A_2.
\end{flalign*}
This amounts to \begin{flalign*}
	&A_2 = \begin{pmatrix}
		M_8 &M_5 &M_2\\ M_7 &M_4 &M_1\\ \pi_0 I_{n-2i} &0 &0
	\end{pmatrix}, \\ 
	&N_1M_8+N_2M_7+\pi_0 N_3=0,\\
	&N_1M_5+N_2M_4-\pi_0 I_i=0,\\
	&N_1M_2+N_2M_1=0, \\
	&N_4M_8+N_5M_7+\pi_0 N_6=0,\\
	&N_4M_2+N_5M_1-\pi_0 I_i=0,\\
	&N_7M_8+N_8M_7+\pi_0 N_9-\pi_0 M_9=0,\\
	&N_7M_5+N_8M_4-\pi_0 M_6=0,\\
	&N_7M_2+N_8M_1-\pi_0 M_3=0.
\end{flalign*}

Similarly, the condition that the natural map $\Lambda_{-i}\otimes_\CO R\ra \Lambda_{i}\otimes_\CO R$ sends $\CF_{-i}$ to $\CF_{{i}}$ is equivalent to that there exists an $n\times n$ matrix $A_1$ such that \begin{flalign*}
	\iota_1\begin{pmatrix}
    0 & &I_{n-i}\\ &N  \\ I_i & &0
\end{pmatrix}= \begin{pmatrix}
    M\\ I_n\\ M'
\end{pmatrix} A_1.
\end{flalign*}
This amounts to \begin{flalign*}
	&A_1=\begin{pmatrix}
		0 &0 &I_{n-2i}\\ N_1 &N_2 &N_3\\ N_4 &N_5 &N_6
	\end{pmatrix},\\ 
	&M_4N_1+M_1N_4=0,\\
	&M_4N_2+M_1N_5-\pi_0 I_i=0,\\
	&M_7+M_4N_3+M_1N_6=0,\\
	&M_6N_1+M_3N_4-N_7=0,\\
	&M_9+M_6N_3+M_3N_6-N_9=0,\\
	&M_5N_1+M_2N_4-\pi_0 I_i=0,\\
	&M_5N_2+M_2N_5=0,\\
	&M_8+M_5N_3+M_2N_6=0.
\end{flalign*}
    
We obtain that, under Condition LM2, equations for Condition LM3 are the following. 
\begin{flalign*}
	&N_1H_iN_5^t+N_2H_iN_4^t+\pi_0 H_i=0,\\
	&N_1H_iN_2^t+N_2H_iN_1^t=0,\\ 
	&N_4H_iN_2^t +N_5H_iN_1^t+\pi_0 H_i=0,\\
	&N_4^tH_iN_1+N_1^tH_iN_4=0,\\
	&N_4^tH_iN_2+N_1^tH_iN_5+\pi_0 H_i=0,\\
	&N_5^tH_iN_2+N_2^tH_iN_5=0,\\
	&N_6^tH_iN_1+N_3^tH_iN_4+H_{n-2i}N_7=0,\\
	&N_3^tH_iN_5+N_6^tH_iN_2+H_{n-2i}N_8=0, \\
	&N_9^tH_{n-2i}+H_{n-2i}N_9+N_6^tH_iN_3+N_3^tH_iN_6=0.
\end{flalign*}
Set \begin{flalign*}
	X=\begin{pmatrix}
		N_1 &N_2\\ N_4 &N_5
	\end{pmatrix}.
\end{flalign*}
Then the first six equations are equivalent to \begin{flalign*}
	XH_{2i}X^t+\pi_0 H_{2i}=X^tH_{2i}X+\pi_0 H_{2i}=0.
\end{flalign*}
We also obtain that the block matrices $N_7$ and $N_8$ are determined by $X$ and the block matrices $N_3,N_6$.

 \setcounter{equation}{2}
 \renewcommand{\theequation}{\thesection.\arabic{equation}}
\subsubsection{Condition LM4} \label{subsubsecLM4}
We will use the following notations:  
\begin{itemize}
        \item Set $\CB\coloneqq \cbra{S\sset [1,2n]\ |\ \# S=n }$.
	\item For an integer $i$, we write $i^*\coloneqq 2n+1-i$. For $S\in\CB$, we write $$S^*\coloneqq \cbra{i^*\ |\ i\in S},\quad S^\perp\coloneqq [1,2n]\backslash S^*.$$
	\item  Let $\sigma_S$ be the permutation on $[1,2n]$ sending $[1,n]$ to $S$ in increasing order and sending $[n+1,2n]$ to $[1,2n]\backslash S$ in increasing order. 
    
    Denote by $\sgn(\sigma_S)\in\cbra{\pm 1}$ the sign of $\sigma_S$. By \cite[Lemma 2.8]{smithling2015moduli}, we have \begin{flalign}
	        \sgn(\sigma_S)=(-1)^{\Sigma S+\lceil n/2\rceil},  \label{eq-SigmaS}
	    \end{flalign}
        see the notation in \S \ref{subsec-notation-intro}.
	\item Set $W\coloneqq \wedge_F^{n}V$. For $S=\tcbra{i_1<\cdots<i_n} \sset\cbra{1,\ldots,2n}$ of cardinality $n$, we write $$e_S\coloneqq e_{i_1}\wedge\cdots\wedge e_{i_n}\in W.$$
	    Note that $(e_S)_{S\in\CB}$ is an $F$-basis of $W$. 
        In the decomposition $W=W_+\oplus W_-$, by \cite[Lemma 7.2]{pappas2009local}, we have
	$$W_{\pm 1}= \Span_F\cbra{e_S\pm\sgn(\sigma_S)e_{S^\perp}\ |\ S\in \CB}.$$ 
         \item For any $\CO$-lattice $\Lambda$ in $V$, set $$W(\Lambda)\coloneqq \wedge^n_\CO\rbra{\Lambda},\ W(\Lambda)_{\pm 1}\coloneqq W_{\pm 1}\cap W(\Lambda).$$ Then $W(\Lambda)$ (resp. $W(\Lambda)_{\pm 1}$) is an $\CO$-lattice in $W$ (resp. $W_{\pm 1}$).	
      \item Denote by $g_1,\ldots,g_{2n}$ the ordered $\CO$-basis $$e_1,\ldots, e_i,\pi_0 e_{i+1},\cdots,\pi_0 e_{2n}$$ of $\Lambda_{-i}$. Then $(g_S)_{S\in\CB}$ forms an $\CO$-basis of $W(\Lambda_{-i})$.
\end{itemize} 

Set $$f_S\coloneqq e_S\pm \sgn(\sigma_S)e_{S^\perp}. $$ Then 
\begin{flalign}
    f_{S^\perp}=\pm\sgn(\sigma_S)f_S.  \label{gSSperp}
\end{flalign}  
We have $f_S=\pi_0^{-d_S}g_S\pm \sgn(\sigma_S)\pi_0^{-d_{S^\perp}}g_{S^\perp}$. Here, for a subset $S$ of $\sbra{1,2n}$, we define $$d_S\coloneqq \# \rbra{S\cap \sbra{2n-i+1,2n}}.$$ Depending on whether $d_S$ or $d_{S^\perp}$ is larger, a suitable multiple of $f_S$ forms an $\CO$-basis of $W(\Lambda_{-i})_\pm$. The basis is given as follows: set
\[h_S\coloneqq \begin{cases}
    \pi_0^{d_S}f_S=g_S\pm \sgn(\sigma_S)\pi_0^{d_S-d_{S^\perp}}g_{S^\perp} \quad  &\text{if $d_S\geq d_{S^\perp}$},\\ \pi_0^{d_{S^\perp}}f_S=\pi_0^{d_{S^\perp}-d_S}g_S\pm \sgn(\sigma_S)g_{S^\perp} &\text{otherwise}.
\end{cases} \] 
We write $S\preccurlyeq S^\perp$ if 
\begin{itemize}
    \item $S\neq S^\perp$, and $S$ is less than $S^\perp$ in lexicographic order\footnote{By convention, elements in $S$ or $S^\perp$ are listed in increasing order.}, or
    \item $S={S^\perp}$, and $\sgn(\sigma_S)=1$ (for the ``$+$" condition) with $\sgn(\sigma_S)=-1$ (for the ``$-$" condition).
\end{itemize}

\begin{lem}
    Set $\CB_0\coloneqq \cbra{S\in\CB\ |\ S\preccurlyeq S^\perp}$. Then \[\cbra{h_S\ |\ S\in\CB_0} \]
    forms an $\CO$-basis of $W(\Lambda_{-i})_\pm$.
\end{lem}
\begin{proof}
    Note that the role of $\CB_0$ is to pick exactly one element from the pair $\cbra{S,S^\perp}$, subject to the condition that $f_S\neq 0$. Hence, the set $$\cbra{f_S\ |\ S\in\CB_0}$$ is an $F$-basis of $W_\pm$.  Observe that an element \[\sum_{S\in\CB_0}c_Sf_S\in W_\pm, \text{\ where $c_S\in F$,} \]  lies in $W(\Lambda_{-i})$ if and only if $c_Sf_S\in W(\Lambda_{-i})$ for all $S\in\CB_0$. The condition that $c_Sf_S\in W(\Lambda_{-i})$ is equivalent to the existence of some $c_S'\in\CO$ such that $c_Sf_S=c_S'h_S$. Hence, the set $\cbra{h_S\ |\ S\in\CB_0}$ forms an $\CO$-basis of $W(\Lambda_{-i})_\pm=W(\Lambda_{-i})\cap W_\pm$. 
\end{proof}

Note that the line $\wedge^n_\CO\CF_{\Lambda_{-i}}$ is generated by \[\sum_{S\in\CB} a_Sg_S,\] where $a_S$ denotes the determinant of the $n\times n$ submatrix formed by selecting the rows indexed by $S$ (in increasing order) from the matrix \[\begin{pmatrix}
    0 & &I_{n-i}\\ &N  \\ I_i & &0
\end{pmatrix}.\]   Then condition LM4 translates to 
\begin{flalign}
    \sum_{S\in\CB}a_Sg_S= \sum_{S\in\CB_0}c_Sh_S  \label{eq38}
\end{flalign}
for some $c_S\in \CO$. By comparing the coefficients, we obtain that \eqref{eq38} is equivalent to 
\begin{flalign*}
    \text{for all $S\in\CB$ with $d_S\leq d_{S^\perp}$, $a_S=\pm\sgn(\sigma_S)\pi_0^{d_{S^\perp}-d_S}a_{S^\perp}$ }.
\end{flalign*}
Over the special fiber $\RM^\naive_{i,k}$, the condition simplifies to 
\begin{flalign*}
    &\text{for $S\in\CB$ with $d_S=d_{S^\perp}$, $a_S=\pm\sgn(\sigma_S)a_{S^\perp}$, and }\\ &\text{for $S\in\CB$ with $d_S<d_{S^\perp}$, $a_S=0$}.
\end{flalign*}
Note that we have \begin{equation}
    \begin{split}
        d_S\leq d_{S^\perp} &\Longleftrightarrow d_S+d_{S^*}\leq i\\ &\Longleftrightarrow \#\rbra{S\cap\sbra{2n-i+1,2n} } + \#\rbra{S\cap \sbra{1,i } }\leq i\\ &\Longleftrightarrow \#\rbra{S\cap \sbra{i+1, 2n-i} }\geq n-i.
    \end{split}  \label{36eq}
\end{equation}
Recall that  \begin{flalign*}
    X\coloneqq \begin{pmatrix}
        N_1 &N_2\\ N_4 &N_5
    \end{pmatrix},
\end{flalign*}
which is a $2i\times 2i$ matrix.

\begin{prop}\label{lem-wedgeN-1}
    The equality $a_S=0$ over $\RU^\pm_{i,k}$ for $S\in\CB$ with $d_S<d_{S^\perp}$ implies that 
    \begin{flalign}
    	\wedge^{i+1}X=0.  \label{eq-wed}
    \end{flalign}
\end{prop}
\begin{proof}
    Since $a_S$ is a minor of the matrix \begin{flalign*}
    	\begin{pmatrix}
    		0 &I_i &0\\ 0 &0 &I_{n-2i}\\ N_1 &N_2 &N_3\\ N_4 &N_5 &N_6\\ N_7 &N_8 &N_9\\ I_{i} &0 &0
    	\end{pmatrix},
    \end{flalign*}
    any minor of $X$ is equal to $a_S$ or $-a_S$ for some $S\in\CB$. Let $S_1$ be the set of row indices of a $(i+1)\times(i+1)$ minor $D$ of $X$. Then we can pick an $S\in \CB$ such that \begin{flalign}
    	S\cap [i+1,2n-i]=S_1\cup [i+1,n-i]  \label{Sint}
    \end{flalign} and $D=\pm a_S$. By \eqref{Sint}, we have \begin{flalign*}
    	\# (S\cap [i+1,2n-i])\geq n-i+1.
    \end{flalign*}
    Hence, by \eqref{36eq}, we have $d_S< d_{S^\perp}$, and the minor $D$ equals zero over $\RU^\pm_{i,k}$. 
\end{proof}

Note that for $[1,2i]$, we have similar definitions as in the beginning of \S \ref{subsubsecLM4}. For example, for a subset $U\sset [1,2i]$ of cardinality $i$, we denote by $\sigma_U$ the permutation of $[1,2i]$ sending $[1,i]$ to $U$ in increasing order and sending $[i+1,2i]$ to the complement $U^c$ in increasing order.

\begin{prop}\label{lem-wedgeN}
    The equality $a_S=\pm\sgn(\sigma_S)a_{S^\perp}$ for $S\in\CB$ with $d_S=d_{S^\perp}$ implies that
    \begin{flalign}
        [U:U'](X)=\pm \sgn(\sigma_U)\sgn(\sigma_{U'})[U^\perp:U^{'\perp}](X)  \label{eq-wedsgn}
    \end{flalign}
    for any subsets $U$ and $U'$ in $[1,2i]$ of cardinality $i$, where $[U:U'](X)$ denotes the minor of $X$ whose rows (resp. columns) are given by $U$ (resp. $U'$) in increasing order. 
\end{prop}
\begin{proof}
    For $S\in \CB$, set \[S_1\coloneqq S\cap [1,i], S_2\coloneqq S\cap [n-i+1,n+i], S_3\coloneqq S\cap \sbra{i^*,2n}. \]
    Denote $r_i\coloneqq \# S_i$ for $i=1,2,3$. Suppose that $S$ contains $\sbra{i+1,n-i}$ and $r_2=i$. Then the condition $d_S=d_{S^\perp}$ implies that  \begin{flalign*}
    	S=S_1\sqcup S_2\sqcup S_3\sqcup\sbra{i+1,n-i} \text{\ and\ } r_1+r_3=i.
    \end{flalign*} 
    Define \[U\coloneqq S_2-(n-i)=\cbra{x-(n-i)\ |\ x\in S_2}\sset [1,2i] .\]
    Define subsets $T_1\sset \sbra{1,i}$ and $T_3\sset \sbra{i+1,2i}$ via \[\sbra{1,i}\backslash T_1=S_3-(2n-i) \text{\ and\ } \sbra{i+1,2i}\backslash T_3=S_1+i. \]
    Here, elements in $T_1$ and $T_3$ are listed in increasing order. We have $\# T_1=r_1$ and $\# T_3=r_3$. Set \[U'\coloneqq T_1\cup T_3\sset \sbra{1,2i}.\] Then $U'$ is a subset of cardinality $i$.
    Using Laplace expansion, we obtain that \[a_S=(-1)^{(n-i)r_1+\frac{r_1(r_1+1)}{2}+\frac{r_3(r_3+1)}{2}+\Sigma S_1+\Sigma S_3}[U:U'](X). \]
    Similarly, we have \[a_{S^\perp}=(-1)^{(n-i)r_1+\frac{r_1(r_1+1)}{2}+\frac{r_3(r_3+1)}{2}+\Sigma S_1^\perp+\Sigma S_3^\perp}[U^\perp: U^{'\perp}](X), \]
    where $S_1^\perp\coloneqq S^\perp\cap\sbra{1,i}$ and $S_3^\perp\coloneqq S^\perp\cap\sbra{i^*,2n}$.
    Note that $\Sigma S_1+\Sigma S_3=\Sigma S_1^\perp +\Sigma S_3^\perp$, and that $\sgn(\sigma_S)=(-1)^{\Sigma S+\lceil n/2\rceil}$ by \eqref{eq-SigmaS}.
    Hence, the equality $a_S=\pm \sgn(\sigma_S)a_{S^\perp}$ is equivalent to \begin{flalign}
        [U:U'](X)=\pm (-1)^{\lceil n/2\rceil +\Sigma S} [U^\perp:U^{'\perp}](X). \label{2.7}
    \end{flalign}
    By construction, we have \begin{flalign*}
        \Sigma S &=\Sigma S_1+\Sigma S_2+\Sigma S_3 +\sum_{j=i+1}^{n-i}j =\frac{(n+1)n}{2}+2n(i-r_1)+\Sigma U-\Sigma U'.
    \end{flalign*}
    Since $\lceil n/2\rceil\equiv n(n+1)/2 (\mod 2)$,
    it follows that \eqref{2.7} is equivalent to \begin{flalign}
        [U:U'](X)=\pm (-1)^{\Sigma U+\Sigma U'}[U^\perp:U^{'\perp}](X).  \label{312}
    \end{flalign}
    As in \eqref{eq-SigmaS} (see \cite[Lemma 2.8]{smithling2015moduli}), we have \[\sgn(\sigma_U)=(-1)^{\Sigma U+\lceil i/2\rceil }. \]
    Thus, the equality \eqref{312} is also equivalent to \begin{flalign*}
        [U:U'](X)=\pm\sgn(\sigma_U)\sgn(\sigma_{U'})[U^\perp:U^{'\perp}](X). 
    \end{flalign*}
    This proves the proposition.
\end{proof}

\begin{remark}\label{rmk-closed}
   \begin{enumerate}
   	\item By Proposition \ref{lem-wedgeN-1} and \ref{lem-wedgeN}, the special fiber $\RU^\pm_{i,k}$ is a closed subscheme\footnote{After proving in \S \ref{spinflat} that $\RR^\pm_{i,k}$ is reduced, we obtain a posteriori that this closed immersion is in fact an isomorphism, see Proposition \ref{lem-Rikreduced}. } of the scheme defined by the equations \eqref{eq-wed} and \eqref{eq-wedsgn}.
   	\item In \S \ref{subsubsecLM4}, we considered the spin condition (LM4$\pm$) only for $\CF_{-i}\sset \Lambda_{-i}\otimes R$. In fact, one can check that if $\CF_{-i}$ satisfies (LM4$\pm$), then so does $\CF_i$; cf. \cite[Proposition 2.4.3]{luo}. 
   \end{enumerate}
\end{remark}

\subsubsection{Defining equations} \label{subsubsec-equations}
Recall that \[X= \begin{pmatrix}
    N_1 &N_2\\ N_4 &N_5
 \end{pmatrix}, \] which is a matrix of size $(2i)\times (2i)$.
By results from \S \ref{subsubsec-lm2} and \S \ref{subsubsec-LM3}, we obtain that the coordinate ring of $\RU_i^\naive$ is isomorphic to
\begin{flalign*}
     \rR^\naive_i &\coloneqq \frac{\CO[N_3,N_6,N_9,X]}{\left(\Centerstack[l]{$(H_{n-2i}N_9+N_9^tH_{n-2i}+N_6^tH_iN_3+N_3^tH_iN_6$,\\ $XH_{2i}X^t+\pi_0 H_{2i}, X^tH_{2i}X+\pi_0 H_{2i}$  } \right) }.
 \end{flalign*}
 The equation $H_{n-2i}N_9+N_9^tH_{n-2i}+N_6^tH_iN_3+N_3^tH_iN_6=0$ implies that $N_9$ is determined by $N_3$, $N_6$, and the entries lying (strictly) above the secondary diagonal of $N_9$.
 We obtain that $\RU^\naive_i$ is isomorphic to 
 \begin{flalign*}
      \Spec \frac{\CO[X]}{(XH_{2i}X^t+\pi_0 H_{2i}, X^tH_{2i}X+\pi_0 H_{2i})} \times_\CO  \BA_\CO^{(n-2i)(n+2i-1)/2},
 \end{flalign*}
 where $\times_\CO$ denotes the fiber product over $\Spec\CO$.
 
 By results in \S \ref{subsubsecLM4},  the special fiber $\RU^\pm_{i,k}$ admits a closed immersion (see Remark \ref{rmk-closed}) into
 \begin{flalign}
     \Spec \RR^\pm_{i,k}\times_k\BA^{(n-2i)(n+2i-1)/2}_k, \text{\ for\ }   
      \rR^\pm_{i,k} \coloneqq \frac{k[X]}{(XH_{2i}X^t, X^tH_{2i}X, \wedge^{i+1}X,\RI^\mp)}.   \label{Rikequation}
 \end{flalign}
 Here, $\RI^\mp$ denotes the ideal generated by $$[S:T](X)\mp \sgn(\sigma_S)\sgn(\sigma_T)[S^\perp:T^\perp](X)$$ 
for all subsets $S,T\sset [1,2i]$ (in increasing order) of cardinality $i$. 

\begin{prop}\label{prop-schulocal}
	For $0\leq \ell\leq i$, there exists an isomorphism \begin{flalign*}
		\RU^\naive_{i,k}\cap S_i(\ell)\simeq \Spec \frac{k[X]}{(XH_{2i}X^t,X^tH_{2i}X,\wedge^{\ell+1}X)}\times_k\BA^{(n-2i)(n+2i-1)/2}_k,
	\end{flalign*}
	where $S_i(\ell)$ is defined in Definition \ref{introdefn-Ml}.
\end{prop}
\begin{proof}
	It remains to show that the condition $\wedge^{\ell+1}(\iota_2:\CF_i\ra \pi_0\inverse\CF_{-i})=0$ in Definition \ref{introdefn-Ml} amounts to $\wedge^{\ell+1}X=0$. By results in \S \ref{subsubsec-LM3}, the matrix corresponding to $\iota_2$ is given by (over $k$) 
    \begin{flalign*}
    	A_2 = \begin{pmatrix}
		M_8 &M_5 &M_2\\ M_7 &M_4 &M_1\\ 0 &0 &0
	\end{pmatrix}.
    \end{flalign*}
    By equations in \S \ref{subsubsec-LM3}, we have \begin{flalign*}
    	\begin{pmatrix}
    		M_8\\ M_7
    	\end{pmatrix} = \begin{pmatrix}
    		M_5 &M_2\\ M_4 &M_1
    	\end{pmatrix}\begin{pmatrix}
    		N_3\\ N_6
    	\end{pmatrix}.
    \end{flalign*}
    Hence, the rank of $A_2$ is equal to that of $\begin{psmallmatrix}
    	M_5 &M_2\\ M_4 &M_1
    \end{psmallmatrix}$, which also equals  $\rk X=\rk \begin{psmallmatrix}
    	N_1 &N_2\\ N_4 &N_5
    \end{psmallmatrix}$ by \S \ref{subsubsec-lm2}.  
\end{proof}

\begin{corollary}\label{coro-redCM}
	Denote $$R(\ell)\coloneqq \frac{k[X]}{(XH_{2i}X^t,X^tH_{2i}X,\wedge^{\ell+1}X)}.$$
	For $0\leq \ell<i$, the ring $R(\ell)_\red$ is a normal domain and Cohen--Macaulay. For $\ell=i$, $R(i)$ has exactly four minimal primes $\fq_d$ ($d=1,2,3,4$), for which the quotient $R(i)/\fq_d$ is a normal domain and Cohen--Macaulay.
\end{corollary}
\begin{proof}
	By Proposition \ref{prop-stratificationMpm} (3), $S_i(i)$ has exactly four irreducible components,  and 
	For $0\leq\ell<i$, $S_i(\ell)_\red$ is a Schubert variety by Proposition \ref{prop-stratificationMpm} (3). Hence, $S_i(\ell)_\red$ (resp. each irreducible component of $S_i(i)$) is irreducible, normal and Cohen--Macaulay; cf. \cite[Theorem 0.3]{PR08}. Then the corollary follows from Proposition \ref{prop-schulocal}. 
\end{proof}

\section{Reducedness of a class of matrix-defined rings} \label{sec-flatness}
In this section, we prove that a class of matrix-defined rings is reduced, completing the proof of Theorem \ref{intro-2}. Although the statement is purely commutative algebra, our argument relies on the representation theory of complex orthogonal groups. This section is self-contained and of independent interest. 

\subsection{} \label{subsec-Rnotation}
Let $N\geq 1$ be an integer\footnote{For applications, we take $N=2i$ for some integer $0\leq i\leq n$. We also include the case of odd $N$ for future use.}. Set $m=\lfloor N/2\rfloor$. Denote \begin{flalign}
	J=J_N\coloneqq \begin{pmatrix}
0 & 1 &        &        &        &        \\
1 & 0 &        &        &        &        \\
  &   & 0 & 1  &        &        \\
  &   & 1 & 0  &        &       \\
  &   &   &    & \ddots &       \\
  &   &   &    &  &0 & 1   \\
  &   &   &    &  &1 & 0   
\end{pmatrix} \text{\ resp.\ }\begin{pmatrix}
0 & 1 &        &        &        &        \\
1 & 0 &        &        &        &        \\
  &   & 0 & 1  &        &        \\
  &   & 1 & 0  &        &       \\
  &   &   &    & \ddots &       \\
  &   &   &    &  &0 & 1   \\
  &   &   &    &  &1 & 0\\ & & & & & & &1   
\end{pmatrix}    \label{Jmatrix}
\end{flalign} 
if $N=2m$ is even, resp. $N=2m+1$ is odd.
Define the ordered set \begin{flalign}
    \sI\coloneqq \begin{cases}
    	\cbra{\ov{1} <1<\ov{2}<2<\cdots<\ov{m}<m} &\text{if $N=2m$;} \\ \cbra{\ov{1} <1<\ov{2}<2<\cdots<\ov{m}<m<0} &\text{if $N=2m+1$.}    
    \end{cases}    \label{eq-sI}
\end{flalign} 
Set $\bar 0\coloneqq 0$ and $\bar\barj\coloneqq j$ for $j\in\sI$. Then the bar operation $\bar{\ }$ is an involution on $\sI$. 
\begin{defn}
	Suppose that $N=2m$. Let $\sB$ denote the set of subsets of $\sI$ of cardinality $m$.  For a subset $U=\cbra{j_1<j_2<\cdots<j_m}\sset \sI$ in $\sB$, we write\footnote{Via a chosen bijection between $\sI$ and $[1,2m]$ (see \S \ref{subsec-Rnotation}), the operation $\perp$ used in \S \ref{subsubsecLM4} agrees with the one used here; we thus keep the same symbol by abuse of notation.} $$U^\perp\coloneqq (\bar U)^c=\cbra{k_1<\cdots<k_m}$$ for the complement of $\bar U=\cbra{\barj_1,\ldots,\barj_m}$. Denote by $\tau_U$ the permutation of $\sI$ which rearranges the sequence $$j_1,j_2,\ldots,j_m,\bar k_1,\bar k_2\ldots,\bar k_m$$ in increasing order from left to right.
\end{defn}

\begin{defn} \label{defn-RNJ}
	Let $R$ be a $\BZ[1/2]$-algebra. Let $X$ be a generic $N\times N$ matrix.
	\begin{enumerate}
		\item For $0\leq \ell\leq m$, set \begin{flalign*}
    {}_J\CR_R(\ell)\coloneqq \frac{R[X]}{(XJX^t,X^tJX,\wedge^{\ell+1}X)}. 
\end{flalign*}
        \item Suppose $N=2m$ is even. Set \begin{flalign*}
    {}_J\CR^\pm_R \coloneqq \frac{R[X]}{(XJX^t,X^tJX,\wedge^{m+1}X)+{}_J\CI^\mp}, 
\end{flalign*}
where ${}_J\CI^\mp$ denotes the ideal generated by $$[U:U'](X)\mp \sgn(\tau_U)\sgn(\tau_{U'})[U^\perp:U^{'\perp}](X)$$ for all $U,U'\in\sB$. By definition, ${}_J\CR_R^\pm$ is a quotient ring of ${}_J\CR_R(m)$. 
        \item Suppose $N=2m$. Set \begin{flalign*}
			{}_J\CR_{R,1}^+\coloneqq {}_J\CR^+_R/\fp_1={}_J\CR_R(m)/\fp_1, \quad {}_J\CR_{R,2}^+\coloneqq {}_J\CR_R^+/\fp_2={}_J\CR_R(m)/\fp_2,
		\end{flalign*}
		where $\fp_1$ denotes the ideal generated by \begin{flalign*}
			[U:U'](X) - \sgn(\tau_U)[U^\perp:U'](X),\quad [U:U'](X)-\sgn(\tau_{U'})[U:{U'}^\perp](X)
		\end{flalign*}
		for all $U,U'\in \sB$; and $\fp_2$ denotes the ideal generated by \[ [U:U'](X) + \sgn(\tau_U)[U^\perp:U'](X),\quad [U:U'](X)+\sgn(\tau_{U'})[U:{U'}^\perp](X)\]
		for all $U,U'\in\sB$.
		
		Clearly, ${}_J\CR_{R,j}^+$ is a quotient ring of ${}_J\CR_R^+$ for $j=1,2$.
	\end{enumerate} 
	
For simplicity, we will write $\CR_R(\ell)$ for ${}_J\CR_R(\ell)$, and write $\CR^\pm_R$ for ${}_J\CR^\pm_R$, and write $\CR^+_{R,j}$ for ${}_J\CR^+_{R,j}$, $j=1,2$. If the base ring $R$ is clear, we further abbreviate to $\CR(\ell)$, $\CR^\pm$ or $\CR^+_j$. 
\end{defn}

\subsection{$SO(N)$-action on $\CR^\pm$ }
Denote by $\GL(N)$ the general linear group over $R$. Then $\GL(N)$ naturally acts on a free $R$-module of rank $N$ with a basis indexed by $\sI$.  Denote by $\RO(N)=\RO(J)$ (resp. $\SO(N)$) the (resp. special) orthogonal group over $R$ attached to the symmetric matrix $J$. We write $GL(N)$, $O(N)$ and $SO(N)$ for the group of $R$-points of the corresponding algebraic groups.  

Note that the polynomial ring $R[X]$ of $N^2$ variables has a natural structure of a $GL(N)$-bimodule: for $g_1,g_2\in GL(N)$ and $f(X)\in R[X]$, define \[g_2f(X)g_1\coloneqq f(g_1Xg_2). \]
By construction, this $GL(N)$-bimodule structure on $R[X]$ induces an $O(N)$-bimodule structure on $\CR(\ell)$ for $0\leq \ell\leq m$.

In the rest of the paper, we will use the following facts about minors in linear algebra. 
\begin{lemma}\label{lem-linearalg}
    \begin{enumerate}
        \item ({\cite[Theorem 2.4.1]{prasolov1994problems}}) Let $A$ be a $d\times d$ matrix for a positive integer $d$. Fix a subset $F=\cbra{i_1,\ldots,i_p}\sset D\coloneqq \cbra{1,2,\ldots,d}$. The Laplace expansion for $F$ is \[\det(A) = \sum_{\overset{U\sset D}{|U|=p}}\varepsilon(F,U) [F:U](A)[D\backslash F:D\backslash U](A), \]
        where $\varepsilon(F,U)\in\cbra{\pm 1}$ depends on the orders of the elements in $F,U$ and their complements $D\backslash F$, $D\backslash U$. 
        \item ({\cite[Theorem 2.3]{prasolov1994problems}}) Let $A$ (resp. $B$) be a $d\times e$ (resp. $e\times d$) matrix. Let $S$ and $S'$ be two subsets of $D=\cbra{1,2,\ldots,d}$ of cardinality $p\leq \min\cbra{d,e}$. Fix an order of the elements in $S$ (resp. $S'$). The Binet-Cauchy formula is \[[S:S'](AB) = \sum_{\overset{U\sset D}{|U|=p}}[S:U](A)[U:S'](B), \]
        where $U$ is written in increasing order. 
        \item \cite[Lemma A.1(e)]{caracciolo2013algebraic} Let $A$ be an invertible matrix of size $d\times d$. Let $S$ and $S'$ be two subsets of $D=\cbra{1,2,\ldots,d}$ of cardinality $p$, where elements are written in increasing order. The Jacobi's identity on minors is \[[S:S'](A)= \det(A)\inverse(-1)^{\Sigma S+\Sigma S'}[D\backslash S':D\backslash S](A\inverse),  \] 
        where $\Sigma S$ and $\Sigma S'$ denote the sum of elements in $S$ and $S'$ respectively. 
        
        In particular, if $d$ is even and $p=d/2$, then,
        by \eqref{eq-SigmaS}, the above identity can also be expressed as \begin{flalign}
        	[S:S'](A)= \det(A)\inverse\sgn(\sigma_S)\sgn(\sigma_{S'})[D\backslash S':D\backslash S](A\inverse).  \label{42}
        \end{flalign} 
    \end{enumerate}
\end{lemma}

\begin{lem}\label{lem-SOnaction}
    The $O(N)$-action on $\CR(m)$ induces an $SO(N)$-action on $\CR^\pm$.
\end{lem}
\begin{proof}
    The only non-trivial thing is to show that for $g_1,g_2\in SO(N)$, we have $g_2{}_J\CI^\pm g_1\sset {}_J\CI^\pm$. Let $S,T\sset \sI$ be of cardinality $m$. Recall that we assume $N=2m$ is even when considering $\CR^\pm$.

    With respect to the new index set $\sI$, the identity \eqref{42} in Lemma \ref{lem-linearalg} (3) for $g_1$ becomes \begin{flalign*}
        [S:T](g_1) = \det(g_1)\inverse\sgn(\tau_S)\sgn(\tau_T)[T^c:S^c](g_1\inverse).
    \end{flalign*}
    Since $\det(g_1)=1$ and $g_1^tJg_1=J$, we have $g_1\inverse=Jg_1^tJ$ and \begin{flalign*}
        [S:T](g_1) &=\sgn(\tau_S)\sgn(\tau_T)[T^c:S^c](Jg_1^tJ)\\ &=\sgn(\tau_S)\sgn(\tau_T)[T^\perp:S^\perp](g_1^t)\\ &= \sgn(\tau_S)\sgn(\tau_T)[S^\perp:T^\perp](g_1).
    \end{flalign*}
    Therefore, \begin{flalign*}
        [S:T](X)g_1 &= [S:T](g_1X) =\sum_{U\in\sB}[S:U](g_1)[U:T](X) \text{\ (by Lemma \ref{lem-linearalg}(2)) }\\ &=\sum_{U\in\sB}\sgn(\tau_S)\sgn(\tau_U)[S^\perp:U^\perp](g_1)[U:T](X).
    \end{flalign*}
    By Lemma \ref{lem-linearalg} (2), we have  \begin{flalign*}
        [S^\perp:T^\perp](X)g_1 &= \sum_{U\in\sB}[S^\perp:U](g_1)[U:T^\perp](X)\\ &=\sum_{U\in\sB} [S^\perp:U^\perp](g_1)[U^\perp:T^\perp](X),
    \end{flalign*}
    where the second equality holds because taking $\perp$ is an involution on $\sB$.
    Set $$f_{S,T}(X)\coloneqq [S:T](X)\mp \sgn(\tau_S)\sgn(\tau_T) [S^\perp:T^\perp](X)\in{}_J\CI^\mp.$$ We obtain that \begin{flalign*}
        f_{S,T}(X)g_1 = \sum_{U\in\sB} \rbra{\sgn(\tau_U)\sgn(\tau_S)[S^\perp:U^\perp](g_1)} f_{U,T}(X) \in{}_J\CI^\mp.
    \end{flalign*}
    Thus, ${}_J\CI^\mp$ is stable under the right action of $SO(N)$. A similar argument shows that it is also stable under the left action of $SO(N)$, which proves the lemma. 
\end{proof}
\begin{remark}\label{rmk+-}
    The proof of Lemma \ref{lem-SOnaction} also implies that any element $g\in O(N)$ with $\det(g)=-1$ takes ${}_J\CI^+$ to ${}_J\CI^-$. It follows that $\CR^+$ is isomorphic to $\CR^-$. 
\end{remark}

\subsection{Basics on tableaux} \label{subsec-tableauxbasics}
We will mostly follow the notation in \cite{cliff2008basis}. Recall the ordered set \begin{flalign*}
    \sI\coloneqq \begin{cases}
    	\cbra{\ov{1} <1<\ov{2}<2<\cdots<\ov{m}<m} &\text{if $N=2m$;} \\ \cbra{\ov{1} <1<\ov{2}<2<\cdots<\ov{m}<m<0} &\text{if $N=2m+1$.}    
    \end{cases} 
\end{flalign*}

\begin{defn}\label{defn-tableaux}
   \begin{enumerate}
       \item A \dfn{partition} $\lambda$ of a positive integer $d$ into $r$ parts is given by writing $d$ as a sum $d=\lambda_1+\lambda_2+\cdots+\lambda_r$ of positive integers where $\lambda_1\geq \lambda_2\geq \cdots\geq \lambda_r$. Denote by $|\lambda|$ the sum $d$. Sometimes we will write $\lambda=(\lambda_1\geq \lambda_2\geq\cdots \geq \lambda_r)$.
       \item  A \dfn{tableau} of shape $\lambda$ (or a \dfn{$\lambda$-tableau}) is a left justified array with $r$ rows, where the $i$-th row consists of $\lambda_i$ entries from the set $\sI$. For example, \[T=\begin{pmatrix}
           \ov{1} &\ov{2}\\ \ov{2} &2\\ 2
       \end{pmatrix} \]  is a tableau of shape $\lambda=(2\geq 2\geq 1)$.
       \item The \dfn{conjugate} $\lambda'$ of a partition $\lambda$ is the partition whose parts are the column lengths of a tableau of shape $\lambda$. We write $\lambda'=(\lambda_1'\geq\lambda_2' \geq\cdots\geq \lambda_{r'}')$ for some integer $r'$.
       \item Denote by $T^i$ the $i$-th column of $T$. Let $|T^i|$ denote the length of the column $T^i$. Then $|T^i|=\lambda_i'$.
       \item Let $T$ and $T'$ be two tableaux of the same shape. We write $$T\prec T'$$ if $T\neq T'$, and in the right-most column in which there is a differing entry, the top-most entry $i$ of $T$ which differs from an entry $i'$ of $T'$ in the same position satisfies $i<i'$.
       \item We define the following partial order on the set of partitions of with most $N$ parts. We write $\lambda=(\lambda_1\geq \cdots\geq \lambda_k)< \mu=(\mu_1\geq \cdots\geq \mu_\ell)$ if $\mu\neq \lambda$ and \begin{itemize}
           \item $|\lambda|<|\mu|$, or
           \item $|\mu|=|\lambda|$ and $\sum_{j=1}^i\lambda_j\leq \sum_{j=1}^i\mu_j$ for all $1\leq i\leq k$. Here, we set $\mu_j=0$ if $j>\ell$.
       \end{itemize}  
   \end{enumerate}
\end{defn}

\begin{defn}\label{defn-onstandard}
    Let $T$ be a tableau. 
    \begin{enumerate}
        \item We say $T$ is \dfn{$GL(N)$-standard} if it has at most $N$ rows, and if the entries in each column are strictly increasing downward, and the entries in each row are non-decreasing from the left to the right.
        \item For $i=1,2\ldots,m$, we let $\alpha_i$ (resp. $\beta_i$) be the number of entries less than or equal to $i$ in the first (resp. second) column of $T$. Let $T(i,j)$ denote the entry in row $i$ and column $j$ of $T$. 

        We say $T$ is \dfn{$O(N)$-standard}  if it is $GL(N)$-standard, and if 
        \begin{enumerate}
            \item $\lambda_1'+\lambda_2'\leq N$,
            \item for each $i=1,2\ldots,m$, we have \begin{itemize}
                \item [(OS 1)] $\alpha_i+\beta_i\leq 2i$;
                \item [(OS 2)] if $\alpha_i+\beta_i=2i$ with $\alpha_i>\beta_i$ and $T(\alpha_i,1)=i$ and $T(\beta_i,2)=\bari$, then $T(\alpha_i-1,1)=\bari$;
                \item [(OS 3)] if $\alpha_i+\beta_i=2i$ with $\alpha_i=\beta_i (=i)$ and if $\bari,i$ occur in the $i$-th row of $T$, with $\bari$ in $T^1$ and $i$ in $T^b$ for some $b\geq 2$, then $T(i-1,b)=\bari$.
            \end{itemize}
        \end{enumerate}
    \end{enumerate}
\end{defn}

\begin{defn}
    \begin{enumerate}
        \item A \dfn{bitableau} $[S:T]$ of shape $\lambda$ consists of two tableaux $S$ and $T$ of the same shape $\lambda$. For example, \[[S:T]=\left[\begin{matrix}
            1 &1\\ 2 &\ov{2}\\ 3
        \end{matrix}\ :\ \begin{matrix}
            \ov{1} &1\\ \ov{2} &2\\ 3
        \end{matrix} \right] \] is a bitableau of shape $(2\geq 2\geq 1)$.  
       \item A bitableau $[S:T]$ is called $GL(N)$-standard (resp. $O(N)$-standard) if both $S$ and $T$ are $GL(N)$-standard (resp. $O(N)$-standard).
       \item Let $R$ be a commutative ring. Let $X$ be an $N\times N$ matrix with columns and rows indexed by $\sI$. For a bitableau \[[S:T]=\left[\begin{array}{c}
            a_1  \\ a_2\\ \vdots\\ a_r
       \end{array}: \begin{array}{c}
            b_1  \\ b_2\\ \vdots\\ b_r
       \end{array} \right] \]
       with one column, we use $[S:T](X)$ to denote the determinant of the $r\times r$ submatrix of $X$ formed by the rows of indices $a_1,\ldots, a_r$ and columns of indices $b_1,\ldots,b_r$, i.e. $$[S:T](X)\coloneqq \sum_{\tau\in S_r}\sgn(\tau)x_{a_1\tau(b_1)}\cdots x_{a_r\tau(b_r)}\in R[X], $$ where $S_r$ denotes the permutation group on the set $\cbra{b_1,\ldots,b_r}$.
       For a general bitableau $[S:T]$ of shape $\lambda$, define \[[S:T](X)\coloneqq \prod_{i=1}^{\lambda_1}[S^i:T^i](X).\]
       We say $[S:T](X)$ is the \dfn{bideterminant} in $R[X]$ attached to $[S:T]$.

       For a quotient ring $R'$ of $R[X]$ and a bitableau $[S:T]$, the bideterminant $[S:T](X)$ maps to an element in $R'$, which is still denoted by $[S:T](X)$ (or simply $[S:T]$).
    \end{enumerate}
\end{defn}

\subsection{Straightening relations}
Writing a bitableau as a linear combination of ``standard" ones is often called \dfn{straightening}. Let $R$ be a $\BZ[1/2]$-algebra. Fix an integer $\ell$ such that $0\leq \ell\leq m$.  Recall that (see Definition \ref{defn-RNJ})
\[  \CR(\ell)= {}_J\CR_R(\ell)= \frac{R[X]}{(XJX^t,X^tJX,\wedge^{\ell+1}X)}.\]

\begin{lemma} \label{lem-L=0}
	Let $S$ and $T$ be two tableaux of the same shape $\lambda$ and of two columns with $\lambda_1'\leq \ell$. Let $a$ be a positive integer less than or equal to $ \lambda_2'(\leq \lambda_1'\leq \ell)$. Suppose that $C$ is a (possibly empty) subset of $\sI$ consisting of $c$ elements, where $c<a$. Let $S_0$ be the tableau obtained from $S$ by deleting its first $a$ rows. Then in $\CR(\ell)$, we have
	\[  L\coloneqq \sum_{\overset{i_1,\ldots,i_a\in\sI\backslash C}{i_1<\cdots<i_a} }\left[
        \begin{array}{cc}
             i_1 & \bari_1 \\ 
              i_2&\bari_2 \\
               \vdots&\vdots\\
              i_a&\bari_a\\
              \multicolumn{2}{c}{S_0}
               \end{array}:T\right]  =0.
	\]
\end{lemma}
\begin{proof}
    Denote by $S_0^1$ (resp. $S_0^2$) the first (resp. second) column of $S_0$ with column length $f$ (resp. $g$). Denote $A=\cbra{i_1,\ldots,i_a}$ and $\ol{A}=\cbra{\bari_1,\ldots,\bari_a}$. By Lemma \ref{lem-linearalg} (1),  we have \begin{flalign*}
        L &=\sum_{\overset{A\sset \sI\backslash C}{|A|=a}}\rbra{\sum_{\overset{U_1\sset T^1}{|U_1|=f}}\varepsilon(S_0^1,U_1)[S_0^1:U_1](X)[A:T^1\backslash U_1](X)}\\ &\hspace{3cm}\times \rbra{\sum_{\overset{U_2\sset T^2}{|U_2|=g}}\varepsilon(S_0^2,U_2)[S_0^2:U_2](X)[\ol{A}:T^2\backslash U_2](X)}\\ &=\sum_{\overset{A\sset\sI\backslash C}{|A|=a}}\rbra{\sum_{\overset{U_1\sset T^1}{|U_1|=f}}\varepsilon(S_0^1,U_1)[S_0^1:U_1](X)[A:T^1\backslash U_1](X)}\\ &\hspace{3cm}\times \rbra{\sum_{\overset{U_2\sset T^2}{|U_2|=g}}\varepsilon(S_0^2,U_2)[S_0^2:U_2](X)[A:T^2\backslash U_2](JX)}\\
        &=\sum_{\overset{U_1\sset T^1}{|U_1|=f}}\sum_{\overset{U_2\sset T^2}{|U_2|=g}}\varepsilon(S_0^1,U_1)\varepsilon(S_0^2,U_2)[S_0^1:U_1](X)[S_0^2:U_2](X)\\ &\hspace{3cm} \times \sum_{\overset{A\sset\sI\backslash C}{|A|=a}}[A:T^1\backslash U_1](X)[A:T^2\backslash U_2](JX).
    \end{flalign*}
    Denote \[f(S_0,U_1,U_2)\coloneqq \varepsilon(S_0^1,U_1)\varepsilon(S_0^2,U_2)[S_0^1:U_1](X)[S_0^2:U_2](X).  \]
    Set $p\coloneqq |\sI\backslash C| =N-c$. Define $$\CP_{\sI\backslash C}$$ the linear operator taking an $N\times N$ matrix $X$ to the $p\times N$ matrix $\CP_{\sI\backslash C}X$ formed by selecting the rows of $X$ whose indices are in $\sI\backslash C$.  We also write $\CP_{\sI\backslash C}$ for the corresponding $p\times N$ matrix. More concretely, $\CP_{\sI\backslash C}$ is the $p\times N$ matrix obtained from the identity matrix $I_N$ by selecting the rows whose indices are in $\sI\backslash C$. For $A\sset \sI\backslash C$, we have \[ [A:T^1\backslash U_1](X)=[A:T^1\backslash U_1](\CP_{\sI\backslash C}X) = [T^1\backslash U_1:A]((\CP_{\sI\backslash C}X)^t), \]
    and \[[A:T^2\backslash U_2](JX) =[A:T^2\backslash U_2](\CP_{\sI\backslash C}JX). \]
    By Lemma \ref{lem-linearalg} (2), we obtain that \begin{flalign*}
        &\sum_{\overset{A\sset\sI\backslash C}{|A|=a}}[A:T^1\backslash U_1](X)[A:T^2\backslash U_2](JX)\\ = &\sum_{\overset{A\sset\sI\backslash C}{|A|=a}}[T^1\backslash U_1:A](X^t\CP_{\sI\backslash C}^t)[A:T^2\backslash U_2](\CP_{\sI\backslash C} JX)\\ =&[T^1\backslash U_1: T^2\backslash U_2](X^t\CP_{\sI\backslash C}^t\CP_{\sI\backslash C}JX).
    \end{flalign*} 
    Observe that \[\CP_{\sI\backslash C}^t\CP_{\sI\backslash C}=I_{\sI\backslash C},\]
    where $I_{\sI\backslash C}$ is the $N\times N$ diagonal matrix whose diagonal entries are $1$ in the positions indexed by $\sI\backslash C$, and $0$ otherwise. Hence,  \[\CP^t_{\sI\backslash C}\CP_{\sI\backslash C}=I_N-I_{C}, \]
    where $I_{C}$ is the diagonal matrix defined similarly as $I_{\sI\backslash C}$.
    Now we obtain that \begin{flalign*}
        L &=\sum_{\overset{U_1\sset T^1}{|U_1|=f}}\sum_{\overset{U_2\sset T^2}{|U_2|=g}}f(S_0,U_1,U_2)[T^1\backslash U_1:T^2\backslash U_2](X^t(I_n-I_C)JX)\\ &= \sum_{\overset{U_1\sset T^1}{|U_1|=f}}\sum_{\overset{U_2\sset T^2}{|U_2|=g}}f(S_0,U_1,U_2)[T^1\backslash U_1:T^2\backslash U_2](-X^tI_CJX) \quad \text{(since $X^tJX=0$ in $\CR(\ell)$)}.
    \end{flalign*}
    As $c<a$, any $a\times a$ minor of $I_C$, and hence of $X^tI_CJX$, is zero in $\CR(\ell)$. This completes the proof of the lemma.  
\end{proof}

Lemma \ref{lem-L=0} is the analogue of the key lemma \cite[Lemma 4.2]{cliff2008basis}. In \loccit, the author proved the statement for the ring \[\frac{K[X]}{(XJX^t-I_N,X^tJX-I_N)}, \] where $K$ is a field of characteristic zero. Unlike the argument in \loccit, our proof works in arbitrary characteristic. 

In the following, by a \dfn{signed sum} of some quantities $x_1,x_2,\ldots,x_r$, we mean a linear combination $\sum_{i=1}^r\varepsilon x_i$ where each $\varepsilon_i\in \cbra{\pm 1}$.

\begin{lemma}\label{lem-os123}
	Let $S$ and $T$ be $\lambda$-tableaux having two columns. 
    \begin{enumerate}
        \item Suppose that $S$ is not $O(N)$-standard violating (OS1) in Definition \ref{defn-onstandard}. Then \begin{flalign}
        	 [S:T]=\sum_{U\in\CS,U\neq S}-[U:T]  \label{eq43}
        \end{flalign} 
        in $\CR(\ell)$,
	where $\CS$ is a set of $\lambda$-tableaux such that $S\prec U$ and $\CS$ does not depend on $T$.
        \item Suppose that $S$ is not $O(N)$-standard violating (OS2). Then $[S:T]$ may be expressed as a signed sum of bideterminants as in \eqref{eq43}.
        \item Suppose that $S$ is not $O(N)$-standard violating (OS3). Then $[S:T]$ can be expressed as one-half \footnote{Note that we assume $2$ is invertible in $\CR(\ell)$.} a signed sum of bideterminants of two types: (i) of shape $\mu$ where $\mu< \lambda$; (ii) of the form $[U:T]$ where $S\prec U$. The tableaux $U$ in (ii) and their signs in the signed sum are independent of $T$.
    \end{enumerate}
\end{lemma}
\begin{proof}
	(1) and (2) follow from the same method as in \cite[Lemma 5.1, 5.2]{cliff2008basis}, with Lemma 4.2 in \loccit\ replaced by Lemma \ref{lem-L=0} here. Note that the expression $L$ appearing here corresponds to $\wt{L}$ in \loccit. In fact, our situation is simpler than in \cite{cliff2008basis}, since we have $L=0$ in $\CR(\ell)$. Similarly, (3) follows from the arguments in \cite[Lemma 5.1]{cliff2008basis}. Note that the proof in \loccit\ also uses \cite[Lemma 3.2]{cliff2008basis}, which remains valid in our setting, since the straightening relation in \cite[Lemma 3.2]{cliff2008basis} lives in $R[X]$, and hence holds in $\CR(\ell)$ as well. 
\end{proof}

\begin{prop}\label{prop-strai}
    Assume $\lambda_1'\leq \ell$. Let $S,T$ be two $\lambda$-tableaux. 
    Then in $\CR(\ell)$, we have \begin{flalign*}
        [S:T]=\sum_{U}a_U[U:T]+s,  \label{ST1}
    \end{flalign*} 
    where the $\lambda$-tableaux $U$ in the sum are $O(N)$-standard, each $a_U\in R$ and is independent of $T$, and $s$ is a linear combination of bideterminants of shapes $\mu<\lambda$. Applying $[S:T]$ to $X^t$, we obtain  \begin{flalign*}
        [T:S]=\sum_{U}a_U[T:U]+s.  
    \end{flalign*} 
\end{prop}
\begin{proof}
    Using Lemma \ref{lem-os123}, the proof of \cite[Theorem 5.1]{cliff2008basis} also works here.  
\end{proof}

\begin{corollary} \label{coro-span}
    Denote \begin{flalign*}
        \CC(\ell)\coloneqq \cbra{O(N)\text{-standard bitableaux } [S:T]\ |\ \text{the length of the first column $\leq \ell$} }  .
    \end{flalign*}  
    The ring $\CR(\ell)$ is spanned by bideterminants in $\CC(\ell)$ as an $R$-module. 
\end{corollary}
\begin{proof}
     It is clear that $R[X]$ is spanned by bideterminants $[S:T]$, where $S,T$ run through all bitableaux of the same shape. Note that by definition of $\CR(\ell)$, bideterminants $[S:T]$ with $|S^1|>\ell$ vanish in $\CR(\ell)$. By repeatedly applying Proposition \ref{prop-strai}, we obtain the corollary. 
\end{proof}

\subsection{Linear independence of bideterminants}
Let $\CR(\ell)_\red$ denote the associated reduced ring of $\CR(\ell)$. By Corollary \ref{coro-span}, the set of bideterminants in $\CC(\ell)$ also spans $\CR(\ell)_\red$.
In this subsection, we will show that the set $\CC(\ell)$ in Corollary \ref{coro-span} forms an $R$-basis of $\CR(\ell)$, and that if $R=\BZ[1/2]$, then $\CR(\ell)_\red=\CR(\ell)$. When no confusion arises, we sometimes simply write $\CR$ for $\CR(\ell)$. 

\begin{lem}\label{lem-gaction}
    Let $[S:T](X)$ be a bitableau of shape $\lambda$ in $R[X]$. For $g\in GL(N)$, we have 
    \[[S:T](X)g = \sum_U[S:U](g) [U:T](X) \] in $R[X]$,
    where $U$ runs through all column-increasing $\lambda$-tableaux. Similarly, we have \begin{flalign*}
    	g[S:T](X) = \sum_U  [U:T](g)[S:U](X)
    \end{flalign*}
    in $R[X]$.
\end{lem}
\begin{proof}
    (cf. \cite[Lemma 6.1]{cliff2008basis}.) The lemma follows directly by applying the Binet-Cauchy formula in Lemma \ref{lem-linearalg} (2) to each column of the bideterminant $[S:T](gX)$.
\end{proof}


\begin{defn} \label{defn-Tlambda}
    Let $\lambda$ be a partition with at most $N$ parts.
    \begin{enumerate}
        \item The \dfn{canonical} $\lambda$-tableau\footnote{Note that $T^\lambda$ in \cite{cliff2008basis} differs from ours. Our definition seems more natural; for example, our $T^\lambda$ is $O(N)$-standard.}, denoted by $T^\lambda$, is the tableau having each entry in row $j$ equal to $\barj\in\sI$. For example, for $\lambda=(2\geq 2\geq 1)$, the canonical $\lambda$-tableau is \[T^\lambda = \begin{pmatrix}
            \ov{1} &\ov{1}\\ \ol{2} &\ov{2}\\ \ov{3}
        \end{pmatrix}.\]
        Note that $T^\lambda$ is $O(N)$-standard.
        \item Denote by $L^\lambda_{\CR_\red}$ (resp. ${}^\lambda L_{\CR_\red}$) the $R$-span of $[T^\lambda:T]$ (resp. $[T:T^\lambda]$) in $\CR(\ell)_\red$ as $T$ runs through all $\lambda$-tableaux. 
    \end{enumerate}
\end{defn}
By Lemma \ref{lem-gaction}, $L_{\CR_\red}^\lambda$ (resp. ${}^\lambda L_{\CR_\red}$) is a left (resp. right) $O(N)$-module. By a similar proof of Corollary \ref{coro-span}, the module $L_{\CR_\red}^\lambda$ (resp. ${}^\lambda L_{\CR_\red}$) is spanned over $R$ by $[T^\lambda:T]$ (resp. $[T:T^\lambda]$), where $T$ runs through all $O(N)$-standard tableaux. 

Let us briefly review results in \cite{king1992construction}. Let $\BE$ be an $N$-dimensional $\BC$-vector space equipped with a non-degenerate symmetric pairing. Assume that there exists a basis $\cbra{w_i}_{i\in\sI}$ such that the symmetric pairing corresponds to the matrix $J$. Let $\lambda$ be a partition of $l$ with at most $N$-parts.  Let $\BE^{\otimes l}$ denote the $l$-fold tensor product of $\BE$. Then $\BE^{\otimes l}$ admits a $\BC$-basis \[\cbra{w_{i_1i_2\ldots i_l}\coloneqq w_{i_1}\otimes w_{i_2}\otimes\cdots\otimes w_{i_l}\ |\ i_k\in\sI \text{\ for\ $1\leq k\leq l$}}. \] The permutation group $S_l$ naturally acts on $\BE^{\otimes l}$ via \[\sigma w_{i_1i_2\ldots i_l}\coloneqq w_{i_{\sigma\inverse(1)}i_{\sigma\inverse(2)}\ldots i_{\sigma\inverse(l)} } \] for $\sigma\in S_l$. The natural left action of $GL(N)$ on $\BE$ induces an action on $\BE^{\otimes l}$.
\begin{defn} \label{415M}
    \begin{enumerate}
        \item Denote by $t^\lambda$ the tableau formed by filling the elements $$1,2,\ldots,l$$ (in this order) consecutively in the arrays of $t^\lambda$ first passing down the leftmost column and then down subsequently columns taken in turn from left to right. For example, if $\lambda=(2\geq 2\geq 1)$, then \[t^\lambda=\begin{pmatrix}
            {1} &4\\ 2 &{5}\\ {3}
        \end{pmatrix}.\]
        \item Define the column group $C^\lambda$ (resp. $R^\lambda$) as the subgroup of $S_l$ which preserves each column (resp. row) of $t^\lambda$. The \dfn{Young symmetriser} $Y^\lambda\in\BZ[S_l]$ is defined as \[Y^\lambda\coloneqq \sum_{\rho\in R^\lambda}\sum_{\sigma\in C^\lambda}\sgn(\sigma)\rho\sigma, \] where $\sgn(\sigma)$ denotes the signature of the permutation $\sigma$.
        \item For each basis element $w=w_{i_1\ldots i_l}\in \BE^{\otimes l}$, we can associate a $\lambda$-tableau $T^\lambda_w$ by replacing each element $j$ in $t^\lambda$ by $i_j$ for $1\leq j\leq l$. For example, for $\lambda=(2\geq 2\geq 1)$ and $w=w_{4\ol{1}2\ol{3}1}$, \[T^\lambda_w=\begin{pmatrix}
            4 &\ol{3}\\ \ol{1} &1\\ 2
        \end{pmatrix}. \]
        This process allows us to identify the set of all $\lambda$-tableaux with a basis of $\BE^{\otimes l}$.  For a $\lambda$-tableau $T$, we also use $T$ to denote the corresponding basis element in $\BE^{\otimes l}$.
        \item Let $M^\lambda\sset \BE^{\otimes l}$ denote the $\BC$-span of all $Y^\lambda T$, as $T$ varies over all $\lambda$-tableaux. This is an irreducible (left) $GL(N)$-module by \cite[Theorem 2.4]{king1992construction}.
    \end{enumerate}
\end{defn}

\begin{example} \label{ex-ylambda}
   \begin{enumerate}
       \item Let $T^\lambda$ be the canonical $\lambda$-tableau defined in Definition \ref{defn-naive} (1). Then $T^\lambda$ is of the form \[\begin{pmatrix}
        \ol{1} &\ol{1} &\cdots &\ol{1}\\ \ol{2} &\ol{2} &\cdots &\vdots\\ \ol{3} &\vdots &\cdots &\ol{\lambda_k'}\\ \vdots &\ol{\lambda_2'} &\cdots \\ \ol{\lambda_1'}
    \end{pmatrix} \] for some integer $k$.  For any $\rho\in R^\lambda$, it follows by definition that $\rho$ fixes $T^\lambda$. Then we obtain that 
    \begin{flalign*}
        Y^\lambda T^\lambda &= |R^\lambda|\sum_{\sigma\in C^\lambda}\sgn(\sigma)\sigma(T^\lambda)\\ &=|R^\lambda|(w_{\ol{1}}\wedge w_{\ol{2}}\wedge\cdots\wedge w_{\ol{\lambda_1'}})\otimes (w_{\ol{1}}\wedge\cdots\wedge w_{\ol{\lambda_2'}})\otimes\cdots\otimes (w_{\ol{1}}\wedge\cdots\wedge w_{\ol{\lambda_k'}}).
    \end{flalign*}
    Here, we identify $\wedge^d\BE$ with a subspace of $\BE^{\otimes d}$ via
    \[\wedge^d\BE\hookrightarrow \BE^{\otimes d}, v_1\wedge\cdots\wedge v_d\mapsto \sum_{\sigma\in S_d}\sgn(\sigma)v_{\sigma(1)}\otimes\cdots\otimes v_{\sigma(d)}. \]
       \item Suppose that $\lambda$ is a partition of $\ell$ with all $\lambda_i=1$. In other words, a $\lambda$-tableau consists of just one column. Then $R^\lambda=\cbra{1}$ and $C^\lambda=S_l$. For any tableau of shape $\lambda$ \[T=\begin{pmatrix}
           i_1\\ i_2\\ \vdots\\ i_l
       \end{pmatrix},\] we have \begin{flalign*}
           Y^\lambda T = \sum_{\sigma\in S_l} \sgn(\sigma) w_{\sigma\inverse(1)\sigma\inverse(2)\ldots\sigma\inverse(l)}=w_{i_1}\wedge w_{i_2}\wedge\cdots \wedge w_{i_l}.
       \end{flalign*}
       Therefore, $M^\lambda=\wedge^l\BE$.
   \end{enumerate} 
\end{example}

\begin{defn} \label{defn-Olambda}
    Let $U\sset \BE^{\otimes l}$ be the span of all tensors of the form \[\sum_{i\in\sI}x\otimes w_i\otimes y\otimes w_{\bari}\otimes z\in \BE^{\otimes l}, \]
    where $x,y$ and $z$ are elements of some (possibly zero) tensor power of $\BE$. This subspace is invariant under the action of $O(N)$. Define \[O^\lambda\coloneqq M^\lambda/(M^\lambda\cap U).\]
\end{defn}

\begin{lem} \label{lem-Onirreducible}
    Assume $\lambda_1'+\lambda_2'\leq N$. The $O(N)$-module $O^\lambda$ is irreducible, and has a $\BC$-basis given by (the image of) $Y^\lambda T$, where $T$ varies over all $O(N)$-standard tableaux. 
\end{lem}
\begin{proof}
    See \cite[Theorem 3.4, 3.9]{king1992construction}.
\end{proof}

\begin{prop}\label{prop-nonzeromap}
    Suppose that $R=\BZ[1/2]$ and $\lambda_1'\leq \ell$. There exists a non-zero homomorphism of left $O(N)$-modules \[F_{O^\lambda}\colon O^\lambda\ra \CR_{\red,\BC}=\CR(\ell)_\red\otimes_R\BC.\]
\end{prop}
\begin{proof}
    (cf. \cite[\S 6]{cliff2008basis}) Denote 
    \[M\coloneqq \wedge^{\lambda_1'}\BE\otimes \wedge^{\lambda_2'}\BE\otimes\cdots \otimes \wedge^{\lambda_k'}\BE. \]
     Then $M$ admits a basis\[ \beta= \cbra{(w_{i_{11}}\wedge w_{i_{12}}\wedge\cdots\wedge w_{i_{1\lambda_1'}})\otimes (w_{i_{21}}\wedge\cdots\wedge w_{i_{2\lambda_2'}})\otimes\cdots\otimes (w_{i_{k1}}\wedge\cdots\wedge w_{i_{k\lambda_k'}}) },  \] 
     where $i_{pq}\in \sI$ for $1\leq p\leq k$ and $1\leq q\leq \lambda_p'$. Set \[m_1\coloneqq (w_{\ol{1}}\wedge w_{\ol{2}}\wedge\cdots\wedge w_{\ol{\lambda_1'}})\otimes (w_{\ol{1}}\wedge\cdots\wedge w_{\ol{\lambda_2'}})\otimes\cdots\otimes (w_{\ol{1}}\wedge\cdots\wedge w_{\ol{\lambda_k'}}) \in M.  \] By Example \ref{ex-ylambda} (1), we have \begin{flalign}
         Y^\lambda T^\lambda=|R^\lambda|m_1.  \label{eq-tlambda}
     \end{flalign}   
     We index $\beta$ by $\tcbra{m_1,m_2,\ldots,m_D}$ by choosing a total order on $\beta$, where $D$ equals the $\BC$-dimension of $M$. 

     For any $\BC$-algebra $A$, an $A$-point of $\Spec \CR_\BC$ is given by an $N\times N$ matrix  $X=(x_{ij})$ with $x_{ij}\in A$ satisfying $XJX^t=X^tJX=0$ and $\wedge^{\ell+1}X=0$. The basis $(w_i\otimes 1)_{i\in\sI}$ of $\BE_A\coloneqq \BE\otimes_\BC A$ identifies $ \BE\otimes_\BC A$ with $A^N$. Hence, the matrix $X$ acts on $\BE_A$, which also induces an action on $M\otimes_\BC A$. For any $x\in M\otimes_\BC A$, write \[Xx=\sum_{i=1}^D a_i(x)(X) (m_i\otimes 1) \] for some (uniquely determined) $a_i(x)(X)\in A$. The association $X\mapsto a_1(x)(X)$ defines a morphism from $\Spec\CR_\BC$ to $\BA^1$, and hence, an element $F_{M,x}$ in $\CR_\BC$. Define \[F_M\colon M\ra \CR_\BC,\quad F_M(x)\coloneqq F_{M,x}.  \]
     By construction, $F_M$ is a homomorphism of (left) $O(N)$-modules such that \[F_M(m_1)=[T^\lambda:T^\lambda](X).\] Note that for any $\lambda$-tableau $T$, we have $Y^\lambda T\in M$. Thus, $M^\lambda$ (see Definition \ref{415M})  is a subspace of $M$. By restriction, we obtain a map \[F_{M^\lambda}\colon M^\lambda\ra \CR_\BC.\] By \eqref{eq-tlambda}, we have \[F_{M^\lambda}(Y^\lambda T^\lambda)= |R^\lambda| F_{M}(m_1) = |R^\lambda|[T^\lambda:T^\lambda]. \]

     For an $A$-valued point $X=(x_{ij})$ of $\Spec\CR_\BC$ and a tensor $$u=\sum_{i\in\sI}x\otimes w_i\otimes y\otimes w_{\bari}\otimes z$$ in $U$ (see Definition \ref{defn-Olambda}), we have \begin{equation}
         \begin{split}
              Xu &= \sum_{i\in\sI} Xx\otimes Xw_i \otimes Xy\otimes Xw_{\bari}\otimes Xz\\ &=\sum_{i\in\sI} Xx\otimes (\sum_{j\in\sI}x_{ji}w_j) \otimes Xy\otimes (\sum_{k\in\sI}x_{k\bari}w_k)\otimes Xz\\ &= \sum_{j,k\in \sI}\rbra{\sum_{i\in\sI}x_{ji}x_{k\bari} }Xx\otimes w_j\otimes Xy\otimes w_k\otimes Xz.   \label{eqXu}
         \end{split}
     \end{equation} 
     Since $XJX^t=0$, we have $\sum_{i\in\sI}x_{ji}x_{k\bari}=0$, and hence $Xu=0$. It follows that the map $F_{M^\lambda}$ factors through $O^\lambda$. Then we obtain a homomorphism of $O(N)$-modules \[F_{O^\lambda}\colon O^\lambda\ra \CR_\BC\] such that $F_{O^\lambda}(Y^\lambda T^\lambda)= F_{M^\lambda}(Y^\lambda T^\lambda)= |R^\lambda|[T^\lambda:T^\lambda]$. By composing with the natural quotient map $\CR_\BC\twoheadrightarrow \CR_{\red,\BC}$, we obtain a homomorphism of $O(N)$-modules \[O^\lambda\ra \CR_{\red,\BC}, \] which we still denote by $F_{O^\lambda}$.    To show that $F_{O^\lambda}$ is non-zero, it suffices to prove that $[T^\lambda:T^\lambda]$ is non-zero in $\CR_{\red,\BC}$.  

     Note that the $N\times N$ diagonal matrix \begin{flalign}
     	X_1\coloneqq I_{\cbra{\ol{1},\ol{2},\ldots,\ol{\ell}}}, \label{X145}
     \end{flalign}   with diagonal entries $1$ in the positions indexed by $\cbra{\ol{1},\ldots,\ol{\ell}}$ and $0$ elsewhere,  defines a $\BC$-point in $\CR_\BC(\BC)=\CR_{\red,\BC}(\BC)$. Since $\lambda_1'\leq \ell$, $[T^\lambda:T^\lambda](X_1)=1\neq 0$. Hence, $[T^\lambda:T^\lambda]$ is non-zero in $\CR_{\red,\BC}$. This completes the proof. 
\end{proof}
\begin{remark}\label{rmk-rightF}
    If we endow $\BE$ with the right $GL(N)$-action, then the module $O^\lambda$ (with the same construction) becomes a right $O(N)$-module. We can similarly define a non-zero homomorphism of right $O(N)$-modules from $O^\lambda$ to $\CR_{\red,\BC}$.
\end{remark}

\begin{corollary}\label{coro-basis}
    Assume that $\lambda_1'\leq \ell$. The left (resp. right) $O(N)$-module $L^\lambda_{\CR_\red}$ (resp. ${}^\lambda L_{\CR_\red}$) is irreducible, and has a $\BC$-basis given by $O(N)$-standard bideterminants $[T^\lambda:T]$ (resp. $[T:T^\lambda]$), where $T$ varies over all $O(N)$-standard tableaux. 
\end{corollary}
\begin{proof}
    Since $O^\lambda$ is irreducible by Lemma \ref{lem-Onirreducible}, the non-zero $O(N)$-homomorphism $F_{O^\lambda}$ in Proposition \ref{prop-nonzeromap} is injective. Since $F_{O^\lambda}(Y^\lambda T^\lambda)=|R^\lambda|[T^\lambda:T^\lambda]\in L^\lambda_{\CR_\red}$, the image of $F_{O^\lambda}$ necessarily lies in $L^\lambda_{\CR_\red}$. By Proposition \ref{prop-strai}, $L^\lambda_{\CR_\red}$ is spanned by $[T^\lambda:T]$, where $T$ varies over $O(N)$-standard tableaux. By Lemma \ref{lem-Onirreducible}, we have \begin{flalign*}
    	\dim_\BC L^\lambda_{\CR_\red}\leq \dim_\BC O^\lambda.
    \end{flalign*} 
    Thus, the injectivity of $F_{O^\lambda}$ implies that $L^\lambda_{\CR_\red}$ is isomorphic to $O^\lambda$, which is irreducible. Furthermore,   the set of $[T^\lambda:T]$ with $T$ being $O(N)$-standard is a $\BC$-basis of $L^\lambda_{\CR_\red}$. 
    
    By Remark \ref{rmk-rightF}, the statement for ${}^\lambda L_{\CR_\red}$ also holds.
\end{proof}

\begin{defn} \label{defn-A}
    Assume that $\lambda_1'\leq \ell$. Let $A(\leq \lambda)$ denote the $\BC$-subspace of $\CR_{\red,\BC}$ spanned by $O(N)$-standard bideterminants of shape $\mu$ with $\mu\leq \lambda$. Define $A(<\lambda)$ similarly. Set \[\ol{A(\leq \lambda)}\coloneqq A(\leq \lambda)/A(<\lambda). \]
    By Lemma \ref{lem-gaction} and Proposition \ref{prop-strai}, $A(\leq\lambda), A(\lambda)$ and $\ol{A(\leq\lambda)}$ are $O(N)$-bimodules. 
\end{defn}

By Corollary \ref{coro-basis}, $L_{\CR_\red}^\lambda\otimes {}^\lambda L_{\CR_\red}$ is an irreducible $O(N)$-bimodule with basis \begin{flalign}
	\cbra{[T^\lambda:T]\otimes [S:T^\lambda]\ |\ \text{$S,T$ are $O(N)$-standard of shape $\lambda$} }.  \label{45}
\end{flalign} 
Then we define a $\BC$-linear map \[\Phi\colon L_{\CR_\red}^\lambda\otimes {}^\lambda L_{\CR_\red}\ra \ol{A(\leq \lambda)}, \text{\ via\ } [T^\lambda:T]\otimes [S:T^\lambda]\mapsto \ol{[S:T]},  \]
for $O(N)$-standard tableaux $S, T$ of shape $\lambda$, and extend $\Phi$ by linearity.

\begin{lem}\label{lem-phiiso}
    The map $\Phi$ is an isomorphism of $O(N)$-bimodules.
\end{lem}
\begin{proof}
    (cf. \cite[Theorem 6.1]{cliff2008basis})
     We first show that $\Phi$ is a homomorphism of $O(N)$-bimodules. Let $g\in O(N)$. By Lemma \ref{lem-gaction}, we have \begin{flalign*}
         [S:T^\lambda]g =\sum_U[S:U](g)[U:T^\lambda],
    \end{flalign*} where $U$ varies over column-increasing $\lambda$-tableaux. By Proposition \ref{prop-strai}, each $[U:T^\lambda]$ can be expressed as \begin{flalign}
         [U:T^\lambda] \equiv \sum_{U'}a_{U,U'}[U':T^\lambda]\ \mod\ A(<\lambda),  \label{auu}
     \end{flalign} 
     where $a_{U,U'}\in \BC$ does not depend on $T^\lambda$, and $U'$ is $O(N)$-standard. Then
     \begin{flalign*}
	     \Phi\rbra{([T^\lambda:T]\otimes [S:T^\lambda])g} &= \Phi\rbra{[T^\lambda:T]\otimes (\sum_U[S:U](g)[U:T^\lambda]) } \\
         &= \Phi\rbra{\sum_{U,U'} a_{U,U'}[S:U](g) [T^\lambda:T]\otimes [U':T^\lambda]   }\\
         &=\sum_{U,U'}a_{U,U'}[S:U](g)\ol{[U':T]}.
	 \end{flalign*}
     On the other hand, by Lemma \ref{lem-gaction},  \begin{flalign*}
         \Phi\rbra{[T^\lambda:T]\otimes [S:T^\lambda]}g &= \ol{[S:T]}g =\sum_U\ol{[S:U](g)[U:T]}.
     \end{flalign*}
     Since $a_{U,U'}$ in \eqref{auu} does not depend on $T^\lambda$, we have \begin{flalign*}
         [U:T] \equiv \sum_{U'}a_{U,U'}[U':T]\ \mod\ A(<\lambda).
     \end{flalign*}
     Thus, \begin{flalign*}
         \Phi\rbra{[T^\lambda:T]\otimes [S:T^\lambda]}g &=\sum_{U}\ol{[S:U](g)\sum_{U'}[U':T]} \\ &=\sum_{U,U'}a_{U,U'}[S:U](g)\ol{[U':T]}.
     \end{flalign*}
     It follows that $\Phi$ is a homomorphism of right $O(N)$-modules. Similarly, we can show that $\Phi$ is also a homomorphism of left $O(N)$-modules. Hence $\Phi$ is a homomorphism of $O(N)$-bimodules.

     Since $L_{\CR_\red}^\lambda$ (resp. ${}^\lambda L_{\CR_\red}$) is an irreducible left (resp. right) $O(N)$-module, the bimodule $L_{\CR_\red}^\lambda\otimes {}^\lambda L_{\CR_\red}$ is also irreducible. As $\Phi$ is clearly non-zero, $\Phi$ is injective. A similar proof of Corollary \ref{coro-span} shows that $\ol{A(\leq \lambda)}$ is spanned by the image of $\Phi$. Therefore, $\Phi$ is an isomorphism of $O(N)$-bimodules. 
\end{proof}

\begin{corollary} \label{coro-Cbasis}
    The set $\CC=\CC(\ell)$ in Corollary \ref{coro-span} (for $R=\BC$) forms a $\BC$-basis of $\CR_{\red,\BC}$.
\end{corollary}
\begin{proof}
    Let $\wt{\leq}$ be a total order refining the partial order $<$ in Definition \ref{defn-tableaux} (6). By induction, $A(\wt{<}\lambda)$ has a basis given by $O(N)$-standard bideterminants in $\CC$ of shape $\wt{<}\lambda$. By \eqref{45} and Lemma \ref{lem-phiiso}, we obtain that $A(\wt{\leq}\lambda)$ has a $\BC$-basis consisting of $O(N)$-standard bideterminants in $\CC$ of shape $\wt{\leq}\lambda$.
    Note that any finite subset of $\CC$ lies in $A(\wt{\leq}\lambda)$ for a sufficiently large $\lambda$. It follows that elements in $\CC$ are linearly independent. Now the lemma follows from Corollary \ref{coro-span}. 
\end{proof}

\begin{thm}\label{thm-basis}
    Let $R$ be a $\BZ[1/2]$-algebra.
    \begin{enumerate}
        \item The set $\CC=\CC(\ell)$ in Corollary \ref{coro-span} is an $R$-basis of $\CR_R=\CR(\ell)_R$. In particular, $\CR_R$ is flat over $R$.
        \item The ring $\CR_{\BZ[1/2]}$ is reduced.
    \end{enumerate}
\end{thm}
\begin{proof}
    We first deal with the case $R=\BZ[1/2]$. Since the reduction map $\CR_{\BZ[1/2]}\ra \CR_{\BZ[1/2],\red}$ is surjective, Corollary \ref{coro-span} implies that there is a short exact sequence of $R$-modules \begin{flalign*}
        0\ra \ker\ra \bigoplus_{[S:T]\in\CC}\BZ[1/2]e_{[S:T]}\twoheadrightarrow \CR_{\BZ[1/2],\red}\ra 0.
    \end{flalign*}
    Since $\BZ[1/2]$ is a PID, the module $\ker$ is a free $\BZ[1/2]$-module. By Corollary \ref{coro-Cbasis}, we have $\ker\otimes_{\BZ[1/2]}\BC=0$. Hence, $\ker=0$ and the surjection \[\bigoplus_{[S:T]\in\CC}\BZ[1/2]e_{[S:T]}\twoheadrightarrow \CR_{\BZ[1/2]}\twoheadrightarrow \CR_{\BZ[1/2],\red} \] is an isomorphism. It follows that $\CR_{\BZ[1/2]}=\CR_{\BZ[1/2],\red}$ is reduced, which proves (2). 

    Let $R$ now be a general $\BZ[1/2]$-algebra. Note that \[\CR_R\simeq \CR_{\BZ[1/2]}\otimes_{\BZ[1/2]}R. \]
    As the set $\cbra{[S:T]\in\CC}$ forms a $\BZ[1/2]$-basis of $\CR_{\BZ[1/2]}$ by the preceding discussion, its base change to $R$ is an $R$-basis of $\CR_R$.    
\end{proof}

\begin{remark}
    The idea of proving that a ring is reduced using an argument as in the proof of Theorem \ref{thm-basis} (2) seems to be unnoticed in the literature. This approach is useful; for example, it circumvents the step to show directly that $\CR_{\BZ[1/2]}$ satisfies Serre's conditions (R0) and (S1), unlike the argument in the proof of \cite[Proposition 5.1]{gortz2003flatness}.
\end{remark}

\subsection{Passing to $\CR^\pm$ and $\CR^+_j$}
We continue to use the notation from the previous subsections. We write $\CR$ for ${}_J\CR_{\BZ[1/2]}$. Assume that $N=2m$ is even.
Recall that (see Definition \ref{defn-RNJ}) we have rings 
\begin{flalign}
    \CR^+ = \CR/\CI^-,\ \CR_j^+= \CR^+/\fp_j,\ j=1,2.
\end{flalign}
By Theorem \ref{thm-basis}, the ring $\CR$ is reduced and flat over $\BZ[1/2]$. We will show that $\CR^+$ and $\CR_j^+$ also have these properties; see Theorem \ref{thm-R+basis} and \ref{thm-R12flat}.

In the following, we will freely view a subset of $\sI$ (see \eqref{eq-sI}) as a tableau with just one column.


\begin{defn}
    Let $S, T$ be two $O(N)$-standard tableaux of shape $\lambda$. \begin{enumerate}
        \item Define a tableau $T^\perp$ as follows: 
           \begin{enumerate}
             \item if $\lambda_1'<m$, then set $T^\perp\coloneqq T$;
              \item if $\lambda_1'=m$, then set $T^\perp$ to be the tableau whose first column is $(T^1)^\perp$ and other columns are the same as $T$. Recall that $T^1$ denotes the first column of $T$.
           \end{enumerate}
       \item Suppose $\lambda_1'=m$. Set $\sgn(\tau_T)\coloneqq \sgn(\tau_{T^1})$.  
    \end{enumerate} 
\end{defn}

\begin{lem}
    If $T$ is $O(N)$-standard of shape $\lambda$, then so is $T^\perp$.
\end{lem}
\begin{proof}
    See \cite[Lemma 4.3]{king1992construction}.
\end{proof}


\begin{lem}\label{lem-sonirr}
    Let $\lambda$ be a partition with $\lambda'_1\leq m$.
    \begin{enumerate}
        \item If $\lambda_1'<m$, then $O^\lambda$ is an irreducible $SO(N)$-module.
        \item If $\lambda_1'=m$, then $O^\lambda$, as an $SO(N)$-module, decomposes into a direct sum of two inequivalent irreducible $SO(N)$-modules, the dimension of each being $\half\dim O^\lambda$. 
    \end{enumerate}
\end{lem}
\begin{proof}
    (1) By the discussion before Definition 4.2 in \cite[\S 4]{king1992construction}, the $O(N)$-module $O^\lambda$ is not self-associate, and hence it is an irreducible $SO(N)$-module. 

    (2) See \cite[Theorem 4.9]{king1992construction}. 
\end{proof}

To have a uniform notation, we set \begin{equation*}
    O^\lambda=O^{\lambda+}=O^{\lambda-}, \text{\ if $\lambda_1'<m$.}
\end{equation*}
For any shape $\lambda$, the base change $\CR_\BC=\CR\otimes_{\BZ[1/2]}\BC$ has an irreducible left $O(N)$-subspace $L^\lambda_\CR$ generated by the bideterminants $[T^\lambda:T]$ for all tableaux $T$ of shape $\lambda$. If $\lambda_1'=m$, we denote by $L_\CR^{\lambda\pm}$ the subspace of $L_\CR^\lambda\sset \CR_\BC$ generated by \[ f_T^\pm\coloneqq [T^\lambda:T]\pm  \sgn(\tau_{T^\lambda})\sgn(\tau_T)[T^\lambda:T^\perp]. \]
If $\lambda_1'<m$, then we set $L_\CR^{\lambda\pm}\coloneqq L_\CR^\lambda$.
\begin{corollary}\label{coro-irrLR}
	Let $\lambda$ be a partition with $\lambda_1'\leq m$. Then $L_\CR^{\lambda\pm}$ is an irreducible $SO(N)$-module.
\end{corollary}
\begin{proof}
	If $\lambda_1'<m$, then $L_\CR^{\lambda\pm}=L_\CR^\lambda\simeq O^\lambda$ is $SO(N)$-irreducible by Lemma \ref{lem-sonirr} (1). Suppose $\lambda_1'=m$. Note that a similar argument in Lemma \ref{lem-SOnaction} implies that $L_\CR^{\lambda\pm}$ is an $SO(N)$-module. Thus, we obtain \begin{flalign*}
		L_\CR^\lambda = L_\CR^{\lambda+}\oplus L_\CR^{\lambda-}
	\end{flalign*}
	as a decomposition of $SO(N)$-modules. By Lemma \ref{lem-sonirr} (2), the space $L_\CR^{\lambda\pm}$ is an irreducible $SO(N)$-module. 
\end{proof}

\begin{defn}
    Let $\lambda$ be a partition with $\lambda_1'=m$. Let $T$ be an $O(N)$-standard tableau of shape $\lambda$. We write $T\preccurlyeq_\pm T^\perp$ if \begin{itemize}
        \item $T\neq T^{\perp}$ and $T^1\leq (T^{1})^\perp$ in lexicographic order, or
        \item $T=T^\perp$ and $\sgn(\tau_T)=\pm \sgn(\tau_{T^\lambda})$.
    \end{itemize}
\end{defn}

\begin{lem}\label{lem-basis+}
    Assume $\lambda_1'=m$.
    The space $L_\CR^{\lambda\pm}$ has a $\BC$-basis given by $$\cbra{f_T^\pm\ |\ \text{ $T$ is $O(N)$-standard of shape $\lambda$ and } T\preccurlyeq_\pm T^\perp}.$$
\end{lem}
\begin{proof}
    Note that the condition $T\preccurlyeq_\pm T^\perp$ picks exactly one element from the pair $\cbra{T,T^\perp}$ when $T\neq T^\perp$, and picks a nonzero $f_T^\pm$ when $T=T^\lambda$. Then the lemma easily follows from the fact that $L^\lambda_{\CR}$ has a $\BC$-basis given by $O(N)$-standard bideterminants $[T^\lambda:T]$ by Corollary \ref{coro-basis}.
\end{proof}
Note that $T^\lambda=T^{\lambda\perp}$ by Definition \ref{defn-Tlambda}. Hence, along the natural quotient map $$\CR_\BC\twoheadrightarrow \CR^+_\BC\twoheadrightarrow \CR^+_{\red,\BC},$$ the space $L_\CR^-$ maps to zero, and the element $f_T^+$ maps to $2[T^\lambda:T]$. 

\begin{defn}
    Let $\lambda$ be a partition with $\lambda_1'\leq m$. Denote \begin{flalign*}
      \text{$L^{\lambda+}_{\CR^+_\red}\coloneqq $ the image of $L_\CR^{\lambda+}$ in $\CR^+_{\red,\BC}$.	}
    \end{flalign*} 
    Then $L^{\lambda+}_{\CR^+_\red}$ is a left $SO(N)$-submodule generated by $[T^\lambda:T](X)$ in $\CR^+_{\red,\BC}$.
\end{defn}

\begin{defn}
    Let $\lambda$ be a partition with $\lambda_1'\leq m$. Let $T$ be an $O(N)$-standard tableau of shape $\lambda$. We say $T$ is \dfn{$SO(N)$-standard} if $\lambda_1'<m$ or $T\preccurlyeq_+ T^\perp$.  

    A bitableau (or bideterminant) $[S:T]$ is $SO(N)$-standard if both $S$ and $T$ are $SO(N)$-standard. 
\end{defn}

\begin{lem}\label{lem-leftirr}
    The space $L_{\CR^+_\red}^{\lambda+}$ is an irreducible left $SO(N)$-module, and has a $\BC$-basis given by $SO(N)$-standard bideterminants $[T^\lambda:T]$ in $\CR_{\red,\BC}^+$. 
\end{lem}
\begin{proof}
    By construction, we have a map \begin{flalign*}
    	f\colon L_\CR^{\lambda+}\ra  L_{\CR_\red^+}^{\lambda+}\ra \CR_{\red,\BC}^+,
    \end{flalign*}
    whose image contains the element $[T^\lambda:T^\lambda](X)$ in $\CR_{\red,\BC}^+$. Using the matrix \eqref{X145} in the proof of Proposition \ref{prop-nonzeromap}, we conclude that $[T^\lambda:T^\lambda](X)$ is nonzero in $\CR_{\red,\BC}^+$. In particular, $f$ is a nonzero map. As $L^{\lambda+}_{\CR}$ is irreducible by Corollary \ref{coro-irrLR}, $f$ is injective, and induces an isomorphism $L_{\CR^+_\red}^{\lambda+}\simeq L_{\CR}^{\lambda+}$; and hence $L_{\CR^+_\red}^{\lambda+}$ is an irreducible $SO(N)$-module.
The description of the $\BC$-basis follows from Lemma \ref{lem-basis+}.
\end{proof}

Similarly, we define ${}^\lambda L^{+}_{\CR^+_\red}$ to be the right $SO(N)$-module generated by $[T:T^\lambda]$ in $\CR_{\red,\BC}^+$, and we have the following lemma.
\begin{lem}\label{L+red}
    The space ${}^\lambda L_{\CR^+_\red}^{+}$ is an irreducible right $SO(N)$-module, and has a $\BC$-basis given by $SO(N)$-standard bideterminants $[T:T^\lambda]$ in $\CR_{\red,\BC}^+$. 
\end{lem}

Similarly as in Definition \ref{defn-A}, we define subspaces $A^+(\leq \lambda), A^+(<\lambda)$ and $\ol{A^+(\leq\lambda)}$  of $\CR^+_{\red,\BC}$. These spaces are $SO(N)$-bimodules.
Under the quotient map $\CR_\BC\twoheadrightarrow \CR^+_{\red,\BC}$, the isomorphism $\Phi$ in Lemma \ref{lem-phiiso} becomes \begin{flalign*}
    \Phi^+\colon L^{\lambda+}_{\CR_{\red}^+}\otimes {}^\lambda L^+_{\CR_{\red}^+}\ra \ol{A^+(\leq\lambda)},\quad [T^\lambda:T]\otimes [S:T^\lambda]\mapsto \ol{[S:T]},
\end{flalign*}
which is an isomorphism of $SO(N)$-bimodules.

\begin{thm}\label{thm-R+basis}
    The ring $\CR^\pm$ is reduced, and flat over $\BZ[1/2]$.
\end{thm}
\begin{proof}
	By Lemma \ref{L+red} and similar arguments as in the case of $\CR$ (cf. Corollary \ref{coro-basis} and Theorem \ref{thm-basis}), we obtain the theorem for $\CR^+$. Since $\CR^-$ is isomorphic to $\CR^+$ (see Remark \ref{rmk+-}), we complete the proof.
\end{proof}

Let $\epsilon= \sgn(\tau_{T^\lambda})=\cbra{\pm}$. For $j=1,2$, along the quotient map \begin{flalign*}
	\CR_\BC\twoheadrightarrow \CR^+_{\red,\BC}\twoheadrightarrow (\CR^+_{j})_{\red,\BC},
\end{flalign*}
the image of $f_T^\epsilon$ in $(\CR_1^+)_{\red,\BC}$ is $2[T^\lambda:T]$, while the image of $f_T^{-\epsilon}$ is zero. The analogous statement for $\CR^+_2$ is obtained by reversing the sign. By similar arguments in Lemma \ref{lem-leftirr}, \ref{L+red} and Theorem \ref{thm-R+basis}, we obtain the following.

\begin{thm}\label{thm-R12flat}
	The rings $\CR^+_1 $ and $\CR^+_2$ are reduced, and flat over $\BZ[1/2]$.
\end{thm}

\subsection{Geometric properties} \label{subsec47}
Let $E$ be a free $\BZ[1/2]$-module of rank $N$. Suppose that with respect to the standard basis $\cbra{\alpha_1,\ldots,\alpha_N}$, $E$ is equipped with the (perfect) symmetric pairing $\varphi$ corresponding to the matrix $J$ in \eqref{Jmatrix}. Denote by $\ud{\End}(E)$ the scheme of endomorphisms of $E$. Then $\ud{\End}(E)$ is isomorphic to $\Spec \BZ[1/2][X]$. In this way, we may view  \begin{flalign*}
	Z_\ell\coloneqq \Spec{}_J\CR(\ell),\ Z^+\coloneqq \Spec{}_J\CR^+,\ Z_j^+\coloneqq \Spec{}_J\CR^+_j,\ \text{for $0\leq\ell\leq m$ and $j=1,2,$}
\end{flalign*}
  as closed subschemes of $\ud{\End}(E)$; see Definition \ref{defn-RNJ}. By Theorem \ref{thm-basis}, \ref{thm-R+basis} and \ref{thm-R12flat}, all these subschemes are reduced, and flat over $\BZ[1/2]$. Fixing the basis $\cbra{\alpha_1,\ldots,\alpha_N}$, we will view an element in $\ud{\End}(E)$ as an $N\times N$ matrix, and vice versa.

Recall that for $0\leq\ell\leq m$, the orthogonal Grassmannian $\OGr(\ell,E)$ consisting of totally isotropic rank $\ell$ subspaces (with respect to $\varphi$) of $E$ is a smooth scheme over $\BZ[1/2]$; cf. \cite[\S 13.3, 13.8]{jantzen2003representations}. Furthermore, the scheme $\OGr(\ell,E)$  is $\BZ[1/2]$-smooth of (relative) dimension $$\ell(N-\ell)-\ell(\ell+1)/2;$$ and it has one (resp. two) connected component(s) if $N>2\ell$ (resp. $N=2\ell$); cf. \cite[Proposition 2.4]{mandelshtam2025positive}. 

\begin{defn}
Recall that $N=2m$. Denote by $\OGr(m,E)_+$ (resp. $\OGr(m,E)_-$) the connected component containing the point corresponding to the maximal totally isotropic rank $m$ subspace $\BZ[1/2]\pair{\alpha_1,\ldots,\alpha_m}$ (resp. $\BZ[1/2]\pair{\alpha_1,\ldots,\alpha_{m-1},\alpha_{m+1}}$).

    Let $Q^+$ be the moduli functor \[\Sch_{/\BZ[1/2]}^\op\ra \Sets\]
    sending a $\BZ[1/2]$-scheme $S$ to the set of triples $(X,L_1,L_2)$, where $X\in \ud{\End}(E)(S)$, and  $L_1, L_2\in \OGr_\pm(m,E)(S)$, such that \begin{itemize}
        \item if $L_1\in \OGr(m,E)_+(S)$, then $L_2\in \OGr(m,E)_{(-1)^m}(S)$;
        \item if $L_1\in\OGr(m,E)_-(S)$, then $L_2\in\OGr(m,E)_{(-1)^{m+1}}(S)$;
        \item $\Im(X)$ (resp. $\Im(JX^tJ)$) is contained in $L_1$ (resp. $L_2$).
    \end{itemize} 
\end{defn}
Denote \[\RO_+\coloneqq \OGr(m,E)_+\times \OGr(m,E)_{(-1)^m} \text{\ and\ } \RO_-\coloneqq \OGr(m,E)_-\times \OGr(m,E)_{(-1)^{m+1}}.\] Then $\RO_\pm$ is irreducible and smooth over $\BZ[1/2]$.
By definition of $Q^+$, we have an obvious diagram \begin{flalign}
    \RO_+\sqcup\RO_- \xleftarrow{\ p_1\ } Q^+\xrightarrow{\ p_2\ } \ud{\End}(E),
      \label{diagQ-1}
\end{flalign}  
where $p_1((X,L_1,L_2))\coloneqq (L_1,L_2)$ and $p_2((X,L_1,L_2))\coloneqq X$.

\begin{lem}\label{lemQ}
    The morphism $p_1: Q^+\ra \RO_+\sqcup\RO_-$ is a vector bundle of rank $m^2$. In particular, $Q^+$ has two connected components and is smooth over $\BZ[1/2]$ of (relative) dimension $m(2m-1)$.
\end{lem}
\begin{proof}
    Note that we have a canonical isomorphism \begin{flalign*}
        \Phi\colon E\simto E^\vee, \quad x\mapsto \varphi(-,x),
    \end{flalign*} where $E^\vee$ denotes the $\BZ[1/2]$-linear dual of $E$.    
    For any $\BZ[1/2]$-algebra $A$, we write $Q^+(A)$ for $Q^+(\Spec A)$. Let $(X,L_1,L_2)\in Q^+(A)$. The composite map  \begin{flalign*}
        E\otimes_{\BZ[1/2]}A\xrightarrow{\Phi} E^\vee\otimes_{\BZ[1/2]}A\xrightarrow{X^\vee}E^\vee\otimes_{\BZ[1/2]}A\xrightarrow{\Phi\inverse} E\otimes_{\BZ[1/2]}A
    \end{flalign*}
    is given by the matrix $X^\ad\coloneqq J\inverse X^tJ=JX^tJ$. The condition that $\Im(X)\sset L_1$ amounts to that $$X\colon E\otimes_{\BZ[1/2]}A\ra E\otimes_{\BZ[1/2]}A$$ factors through $\iota_1\colon L_1\hookrightarrow E\otimes_{\BZ[1/2]}A$. Note that the inclusion $L_2\hookrightarrow E\otimes_{\BZ[1/2]}A$ induces a surjection \[q_2\colon E\otimes_{\BZ[1/2]}A\xrightarrow{\Phi}E^\vee\otimes_{\BZ[1/2]}A\twoheadrightarrow L_2^\vee, \]
    whose kernel is the orthogonal complement $L_2^\perp\sset E\otimes_{\BZ[1/2]}A$. For any $m\in L_2^\perp$, we have \begin{flalign*}
        \varphi(E,X(m)) = \varphi(X^\ad (E),m).
    \end{flalign*}
    Since $X^\ad(E)\sset L_2$, we have \begin{flalign*}
        \varphi(E,X(m))=\varphi(E,X(m)) \sset \varphi(L_2,m) =\cbra{0}
    \end{flalign*}
    Since $\varphi$ is perfect, $X(m)=0$. It follows that the map \[E\otimes_{\BZ[1/2]}A\xrightarrow{X}L_1\sset E\otimes_{\BZ[1/2]}A \]
    factors through $q_2$. We obtain a map $L_2^\vee\ra L_1$. Conversely, given a map $f\colon L_2^\vee\ra L_1$, we can construct a map
    \[\wt{f}\colon E\otimes_{\BZ[1/2]}A\xrightarrow{q_2}L_2^\vee \xrightarrow{f}L_1\xrightarrow{\iota_1}E\otimes_{\BZ[1/2]}A. \]
    Clearly, $\wt{f}$ satisfies that $\Im(\wt{f})\sset L_1$ and $\Im(J{\wt{f}}^tJ)\sset L_2$.

    Denote by $\CL_\pm$ the universal maximal totally isotropic subspace over $\OGr(m,E)_\pm$. Then $\CL_\pm$ is a vector bundle of rank $m$ over $\OGr(m,E)_\pm$. Then $\CL_+^{1}\times \CL_{(-1)^m}^2$ is a pair of universal maximal totally isotropic subspace over $\RO_+$. Here, $\CL_+^1$ (resp. $\CL_{(-1)^m}^2$) denotes the pullback of $\CL_+$ (resp. $\CL_{(-1)^m}$) to $\RO_+$. A similar statement holds for $\RO_-$. Then by the previous discussion, we deduce that \begin{flalign}
        Q^+ = \CHom(\CL_{(-1)^m}^{2\vee},\CL_+^1)\sqcup \CHom(\CL_{(-1)^{m+1}}^{2\vee},\CL_-^1),
    \end{flalign}
    where $\CHom$ denotes the scheme of homomorphisms of vector bundles. Since both vector bundles $\CHom(\CL_{(-1)^m}^{2\vee},\CL_+^1)$ and $\CHom(\CL_{(-1)^{m+1}}^{2\vee},\CL_-^1)$ have rank $m^2$, we obtain the lemma. 
\end{proof}

\begin{lem}\label{lem-facZ}
    The morphism $p_2$ in \eqref{diagQ-1} factors through $Z^+\sset \ud{\End}(E)$.
\end{lem}
\begin{proof}
   Since $Q^+$ is reduced by Lemma \ref{lemQ}, we may check for the geometric points. 
   Let $(X,L_1,L_2)$ be a geometric point in $Q^+(\kappa)$ for an algebraically closed field $\kappa$.
   the condition that $\Im(X)$ lies in $L_1$ implies that the column vectors in $X$ are isotropic. Then we have $X^tJX=0$. Similarly, the condition $\Im(JX^tJ)\sset L_2$ implies that $XJX^t=0$.

   Denote $\ol{E}\coloneqq  E\otimes_{\BZ[1/2]}\kappa$. The map $X$ induces a map \[\wedge^mX\colon \wedge^m(\ol{E})\ra \wedge^m(\ol{E}),  \]
   which factors through $\wedge^mL_1$. Recall that (cf. \cite[\S 7.1.3]{pappas2009local}) we have a decomposition \begin{flalign*}
   	\wedge^m(\ol{E}) = \wedge^m(\ol{E})_+\oplus \wedge^m(\ol{E})_-,
   \end{flalign*} 
   where $\wedge^m(\ol{E})_\pm$ is generated by \[\cbra{\beta_S\coloneqq \alpha_S\pm\sgn(\tau_S)\alpha_{S^\perp} \ |\ S\in\sB}. \]
   Here, we reindex the ordered set $\cbra{1,2,\ldots,n}$ by $\sI$ as in \eqref{eq-sI}.
   
   Suppose $L_1\in\OGr(m,E)_+(\kappa)$.  We obtain that $\Im(\wedge^mX)$ lies in ${\wedge^m(\ol{E})}_+$; cf. \cite[\S 7.1.4]{pappas2009local}. 
   Thus, for any $S\in\sB$, \[(\wedge^mX)(\alpha_S)= \sum_{T\in\sB}[S:T](X)\alpha_T \in \wedge^m(\ol{E})_+\sset \wedge^m(\ol{E}). \]
   Using that $\cbra{\alpha_S}_{S\in\sB}$ forms a $\kappa$-basis of $\wedge^m(\ol{E})$,  we obtain \begin{flalign}
       [S:T](X) = \sgn(\tau_T)[S:T^\perp](X).  \label{214}
   \end{flalign}
   Similarly, the condition that $\Im(JX^tJ)\sset L_2\sset \OGr(m,E)_{(-1)^m}$ translates to \begin{flalign}
       [S:T](JX^tJ)=(-1)^m\sgn(\tau_T)[S:T^\perp](JX^tJ) \label{213}
   \end{flalign} 
   for any $S,T\in\sB$.  Note that \begin{flalign*}
       [S:T](JX^tJ) = [\ol{S}:\ol{T}](X^t)=[\ol{T}:\ol{S}](X).
   \end{flalign*}
   By construction, we have 
    \[\sgn(\tau_T)=(-1)^{m}\sgn(\tau_{\ol{T}}). \]
   Then we obtain that \eqref{213} is equivalent to \begin{equation}
       \begin{split}
           [\ol{T}:\ol{S}](X) &= (-1)^m(-1)^m\sgn(\tau_{\ol{T}})[\ol{T}^{\perp}:\ol{S}](X) \text{\ for any $S,T\in\sB$} \\  \Longleftrightarrow [S:T](X) &=\sgn(\tau_S)[S^\perp:T](X) \text{\ for any $S,T\in\sB$.}  \label{215}
       \end{split}
   \end{equation} 
   By \eqref{214} and \eqref{215}, we deduce that \begin{equation*}
       [S:T](X)=\sgn(\tau_T)[S:T^\perp](X)=\sgn(\tau_T)\sgn(\tau_S)[S^\perp: T^\perp](X).
   \end{equation*}
   Thus, $X\in Z^+$. If $L_1\in \OGr(m,E)_-$, the proof is completely analogous. 
\end{proof}

Now we see that the diagram \eqref{diagQ-1} induces \begin{flalign}
    \RO_+\sqcup\RO_- \xleftarrow{\ p_1\ } Q^+\xrightarrow{\ p_2\ } Z^+\sset \ud{\End}(E).  \label{diagQ-2}
\end{flalign}  
Set \[Q^+_1\coloneqq p_1\inverse(\RO_+),\quad Q^+_2\coloneqq p_1\inverse(\RO_-). \] Then \[Q^+=Q^+_1\sqcup Q^+_2\] is the decomposition of $Q^+$ into irreducible components. By the proof of Lemma \ref{lem-facZ}, we have \begin{flalign*}
	 p_2(Q_1^+)\sset Z_1^+\text{\ and\ } p_2(Q_2^+)\sset Z_2^+.
\end{flalign*}  

For any $X\in Z^+(\kappa)$, then $\Im(X)$ is a totally isotropic subspace in $\ol{E}$. Hence, there exists a maximal totally isotropic subspace $L_1$ containing $\Im(X)$. Similarly, we have a subspace $L_2\in\OGr(m,\ol{E})$ containing $\Im(X^\ad)$. Assume that $L_1\in \OGr(m,\ol{E})_+$. By the proof of Lemma \ref{lem-facZ}, the formula \eqref{214} holds for $X$. Since $X\in Z^+$, we obtain that \eqref{213} holds for $X$. Hence, we have $L_2\in\OGr(m,\ol{E})_{(-1)^m}$. Similarly, if $L_1\in \OGr(m,\ol{E})_-$, then $L_2\in \OGr(m,\ol{E})_{(-1)^{m+1}}$. In particular, we have $X\in Z_1^+\cup Z_2^+$ and $(X,L_1,L_2)\in Q^+$.  Therefore, $p_2: Q^+\ra Z^+$ is surjective and \begin{flalign*}
	 p_2(Q_1^+)= Z_1^+, p_2(Q_2^+)= Z_2^+ \text{\ and\ } Z^+=Z_1^+\cup Z_2^+.
\end{flalign*}  
Since $Q_1^+$ and $Q_2^+$ are irreducible, \begin{flalign}
	Z^+=Z_1^+\cup Z_2^+  \label{Zdecom}
\end{flalign}  is the irreducible decomposition of $Z^+$.

Note that over the open dense locus $\cbra{\rk X=m}$ of $Z^+_j$ ($j=1,2$), the morphism $p_2$ is an isomorphism ($L_1=\Im(X)$ and $L_2=\Im(X^\ad)$ are uniquely determined by $X$). In particular, we have the following.
\begin{prop}\label{prop-zjgensm}
	 For $j=1,2$, the scheme $Z^+_j$ is irreducible and generically smooth over $\BZ[1/2]$ of relative dimension $m(2m-1)$. Moreover, each geometric fiber of $Z^+_j\ra \Spec\BZ[1/2]$ is irreducible.  
\end{prop}
\begin{proof}
	Let $R=\BZ[1/2]$ or an algebraically closed field of characteristic different from $2$. By the base change to $R$, the diagram \eqref{diagQ-2} induces \begin{flalign*}
    (\RO_+\otimes R) \sqcup (\RO_-\otimes R) \xleftarrow{\ p_1\ } Q^+\otimes R \xrightarrow{\ p_2\ } Z^+\otimes R.  
\end{flalign*}  
All the tensors are over $\BZ[1/2]$. Note that $\RO_\pm\otimes R$ is also irreducible and smooth over $R$. The same arguments as in the case $R=\BZ[1/2]$ imply that \begin{flalign}
	 Z^+\otimes R = (Z_1^+\otimes R)\cup (Z_2^+\otimes R) \label{ZRdecom}
\end{flalign} is the irreducible decomposition of $Z^+\otimes R$, and that $Z_j^+\otimes R$ has the properties as in the proposition.
\end{proof}

\begin{defn}
	Assume $N\geq 1$ and $0\leq \ell\leq m$ with $\ell\neq N/2$. Let $Q_\ell$ be the moduli functor sending any $\BZ[1/2]$-scheme $S$ to the set of triples $(X,L_1,L_2)$, where $X\in\ud{\End}(E)(S)$, $L_1,L_2\in \OGr(\ell,E)(S)$ such that $\Im(X)\sset L_1$ (resp. $\Im(X^\ad)\sset L_2$). 
\end{defn}
Note that $\OGr(\ell,E)$ is irreducible for $0\leq \ell\leq m$ with $\ell\neq N/2$. By similar arguments for $Q^+$, we have a diagram \begin{flalign*}
	\OGr(\ell,E)\times \OGr(\ell,E)\xleftarrow{\ p_1\ } Q_\ell \xrightarrow{\ p_2\ } Z_\ell\sset \ud{\End}(E),
\end{flalign*} 
and the following proposition holds.

\begin{prop}\label{prop-Zlreduced}
	The scheme $Z_\ell$ is irreducible and smooth (resp. generically smooth) of relative dimension $\ell(2N-2\ell-1)$ over $\BZ[1/2]$. Moreover, each geometric fiber of $Z_\ell\ra \Spec\BZ[1/2]$ is irreducible. 
\end{prop}

\subsection{Relation to Schubert varieties and reducedness results}  \label{subsec48}
Set \begin{flalign*}
	\sR(\ell) &\coloneqq \frac{\BZ[1/2][X]}{(XH_NX^t,X^tH_NX,\wedge^{\ell+1}X)};\\
	\text{and\ } \sR^\pm &\coloneqq \frac{\BZ[1/2][X]}{(XH_NX^t,X^tH_NX,\wedge^{m+1}X)+{}_{H}\rI^\mp} (\text{for $N=2m$}),
\end{flalign*}
where ${}_H\rI^\pm$ is defined as in \eqref{Rikequation}.

\begin{lemma}\label{twoRisom}
	We have $\sR(\ell)\simeq {}_J\CR(\ell)$ and $\sR^\pm\simeq {}_J\CR^\pm$.
\end{lemma}
\begin{proof}
	Denote by $C\coloneqq C_N$ the $N\times N$ matrix corresponding to the permutation $\sigma$ of $\sbra{1,N}$ given by $\sigma(j)=2j-1$ for $1\leq j\leq \lceil N/2\rceil$ and $\sigma(j)=2(N+1-j)$ for $\lceil N/2\rceil +1\leq j\leq N$. Then we have $H_N=C^tJ_NC$, and $C$ is the transition matrix changing the standard basis $\cbra{e_1,\ldots,e_N}$ to the new basis $\tcbra{e_j'}_{j\in\sI}$, where
\begin{flalign*}
	e'_{\barj}\coloneqq e_{2j-1} \text{\ and\ } e_j'=e_{2j} \text{\ for $1\leq j\leq m$}  \quad &\text{if $N=2m$};\\ \text{and,\ } e'_{\barj}\coloneqq e_{2j-1} \text{\ and\ } e_j'=e_{2j}\text{\ for $1\leq j\leq m$}, e_0\coloneqq e_N  &\text{\ if $N=2m+1$}.
\end{flalign*} 
The lemma follows by substituting $X$ by $C^tXC$.
\end{proof}

\begin{corollary}\label{coro-red46}
	For $0\leq \ell<m$ and $j=1,2$, denote $Z_{\ell,k}\coloneqq Z_\ell\otimes_{\BZ[1/2]}k$ and $Z^+_{j,k}\coloneqq Z^+_j\otimes_{\BZ[1/2]}k$. Then $(Z_{\ell,k})_\red$ and $(Z_{j,k}^+)_\red$ are irreducible, normal and Cohen--Macaulay.
\end{corollary}
\begin{proof}
	By Proposition \ref{prop-zjgensm} and \ref{prop-Zlreduced}, $Z_{\ell,k}$ and $Z_{j,k}^+$ are irreducible. Denote \begin{flalign}
		Z_{k}^+\coloneqq \Spec {}_J\CR^+\otimes_{\BZ[1/2]}k.  \label{eq-zk+20}
	\end{flalign} By \eqref{ZRdecom}, \begin{flalign*}
		Z_k^+ = Z_{1,k}^+\cup Z_{2,k}^+
	\end{flalign*} is the irreducible decomposition of $Z_k^+$. By Lemma \ref{twoRisom} and Corollary \ref{coro-redCM}, we complete the proof.  
\end{proof}

\begin{prop}\label{prop-zsred}
	For $0\leq \ell<m$ and $j=1,2$, the schemes $Z_{\ell,k}$ and $Z^+_{j,k}$ are reduced. In particular, they are irreducible, normal and Cohen--Macaulay by Corollary \ref{coro-red46}.
\end{prop}
\begin{proof}
	By Theorem \ref{thm-basis} and Theorem \ref{thm-R12flat}, the schemes $Z_\ell$ and $Z_j^+$ are reduced and flat over $\BZ[1/2]$. By Proposition \ref{prop-zjgensm} and Proposition \ref{prop-Zlreduced}, the schemes $Z_{\ell,k}$ and $Z_{j,k}^+$ obtained by base change are generically smooth, hence generically reduced. By Corollary \ref{coro-red46}, $(Z_{\ell,k})_\red$ and $(Z_{j,k}^+)_\red$ are  irreducible and normal. By Hironaka's lemma (\cite[Corollary 3]{kollar1995flatness}), we obtain that the special fibers $Z_{\ell,k}$ and $Z_{j,k}^+$ are reduced.
\end{proof}

Next we will show that the special fiber $Z_k^+$ (see \eqref{eq-zk+20}) is also reduced. We begin with the following proposition, which may be viewed as a variant of Hironaka's lemma. 
\begin{prop}\label{prop-reducedZ}
	Let $\RY\ra \Spec D$ be a locally Noetherian flat scheme over a Dedekind domain $D$. Denote by $\RY_\eta$ (resp. $\RY_s$) the fiber of $\RY$ over the generic point (resp. a closed point $s$) of $D$. Assume that \begin{enumerate}
	    \item[(C1)] $\RY$ is reduced and consists of two irreducible components $\RY_1$ and $\RY_2$, and the schematic intersection $\RY_1\cap \RY_2$ is integral and has codimension one in both $\RY_1$ and $\RY_2$; 
		\item[(C2)] $\RY_s$ has exactly two irreducible components $(\RY_1)_s$ and $(\RY_2)_s$, and the schematic intersection $(\RY_1\cap \RY_2)_s=(\RY_1)_s\cap (\RY_2)_s$ is Cohen--Macaulay;
		\item[(C3)] for $j=1,2$, $(\RY_j)_s$ is generically reduced, and the reduction $(\RY_j)_{s,\red}$ is normal and Cohen--Macaulay.
	\end{enumerate}  
	Write $D_s$ for the localization of $D$ at $s$. Then $\RY\otimes_DD_s$ is Cohen--Macaulay with reduced fiber $\RY_s$.  
\end{prop}
\begin{proof}
	Suppose $j=1$ or $2$. Note that $\RY_j$ is flat over $D$, as $\RY_j$ is integral and $\RY_j\ra \Spec D$ is dominant; see \cite[Proposition 14.14]{gortz2020algebraic}. Similarly, $\RY_1\cap \RY_2$ is also flat over $D$.  Applying Hironaka's lemma \cite[Corollary 3]{kollar1995flatness} to $\RY_j$ (by condition (C3)), we obtain that $(\RY_j)_s=(\RY_j)_{s,\red}$ is reduced and Cohen--Macaulay. Since $\RY_j$ is $D$-flat and $D$ is a Dedekind domain, we obtain that $\RY_j\otimes_DD_s$ is also Cohen--Macaulay by \cite[0C6G]{stacks-project}. Similarly, by (C2), $(\RY_1\cap \RY_2)\otimes_DD_s$ is Cohen--Macaulay. Since $\RY$ is reduced by (C1), $\RY\otimes_DD_s$ is reduced and $\RY\otimes_DD_s$ is the schematic union $(\RY_1\otimes_DD_s)\cup (\RY_2\otimes_DD_s)$. By (C1) and \cite[Lemma 4.22]{gortz2001flatness}, $\RY\otimes_DD_s$ is Cohen--Macaulay. Since $\RY$ is flat over $D$, the special fiber $\RY_s$ of $\RY\otimes_DD_s$ is also Cohen--Macaulay. By (C3), $\RY_s$ is generically reduced. Then we obtain that $\RY_s$ is reduced by Serre's criterion of reducedness \cite[0344]{stacks-project}.
\end{proof}

\begin{coro}\label{prop-Zk+red}
	The scheme $Z^+=\Spec{}_J\CR^+$ is Cohen--Macaulay and the special fiber $Z^+_{k}$ is reduced.
\end{coro}
\begin{proof}
	We prove this by applying Proposition \ref{prop-reducedZ} to $\RY=Z^+$ over $D=\BZ[1/2]$. By Theorem \ref{thm-R+basis}, $Z^+$ is reduced and flat over $\BZ[1/2]$. By \eqref{Zdecom}, $Z^+$ consists of two irreducible components $Z_1^+$ and $Z_2^+$. By Propositions \ref{prop-zjgensm} and \ref{prop-zsred}, and faithfully flat descent, the special fibers $Z_{1}^+\otimes_{\BZ[1/2]}\BF_p$ and $Z_{2}^+\otimes_{\BZ[1/2]}\BF_p$ are normal and Cohen--Macaulay. By construction, we have \begin{flalign*}
		Z_1^+\cap Z_2^+ = \Spec \frac{\BZ[1/2][X]}{(XJX^t,X^tJX,\wedge^{m}X)} = Z_{m-1},
	\end{flalign*}
	which is integral by Theorem \ref{thm-basis} and Proposition \ref{prop-Zlreduced}. By the dimension formula in Proposition \ref{prop-Zlreduced} and Proposition \ref{prop-zjgensm}, $Z_1^+\cap Z_2^+$ has codimension one in both $Z_1^+$ and $Z_2^+$. By Proposition \ref{prop-zsred} and faithfully flat descent, we have that $Z_{m-1}\otimes_{\BZ[1/2]}\BF_p=(Z_1^+\cap Z_2^+)\otimes_{\BZ[1/2]} \BF_p$ is Cohen--Macaulay. Now all the conditions in Proposition \ref{prop-reducedZ} are satisfied for $Z^+$, and hence the localization $Z^+\otimes_{\BZ[1/2]}(\BZ[1/2])_{(p)}$ (for all odd $p$) is Cohen--Macaulay with reduced special fiber $Z^+_{\BF_p}=Z^+\otimes_{\BZ[1/2]}\BF_p$. Then we obtain that $Z^+$ is Cohen--Macaulay and that $Z^+_k=(Z^+_{\BF_p})\otimes_{\BF_p}k$ is reduced (see \cite[030U]{stacks-project}).
\end{proof}

Proposition \ref{prop-zsred} and Corollary \ref{prop-Zk+red}, together with Lemma \ref{twoRisom}, implies Theorem \ref{intro-2}.

\section{Flatness of spin local models}\label{spinflat}

In this section, we prove that $\RM^\pm_\CL$ is $\CO$-flat. This implies that $\RM^\pm_\CL$ is normal, Cohen--Macaulay, $\CO$-flat with reduced special fiber (see Theorem \ref{yangthm} (1)).

\subsection{Pseudo-maximal parahoric case}
By Theorem \ref{introthm-conjpara}, we may assume that $\CL$ is the standard lattice chain $\Lambda_I$ for some non-empty $I\sset [0,n]$. In this subsection, we focus on the case $I=\cbra{i}$. We may further assume $2i\leq n$. By Corollary \ref{coro-ifreduced} and \ref{coro-Ureduced}, it remains to show that $\RU_{i,k}^\pm$ is reduced.  

Recall that, in \S \ref{subsubsec-equations}, we have closed immersions \begin{flalign*}
    	\RU^\pm_{i,k}\hookrightarrow Y^\pm\hookrightarrow \RU^\naive_{i,k},
    \end{flalign*}
    where $Y^\pm\coloneqq \Spec\RR^\pm_{i,k}\times_k\BA_k^{(n-2i)(n+2i-1)/2}$ for the ring $\RR^\pm_{i,k}$ defined by \eqref{Rikequation}.

\begin{prop}\label{lem-Rikreduced} 
    \begin{enumerate}
    	\item  The ring $\RR^\pm_{i,k}$ is reduced.
    	\item  $\RU^\pm_{i,k}\simeq Y^\pm$, and $\RM^\pm_{i,k}$ is reduced. 
    \end{enumerate}
\end{prop}
\begin{proof}
    (1) It immediately follows from Theorem \ref{intro-2} (2).
    
    (2) If $i=0$, then the statement is trivial. Suppose $i\neq 0$.  
    By Proposition \ref{prop-stratificationMpm} and \eqref{ZRdecom}, $\RU^\pm_{i,k}$ and $Y^\pm$ both have two irreducible components and $\RU^\naive_{i,k}=\RU^+_{i,k}\cup \RU^-_{i,k}$. Hence, the closed immersion $\RU^\pm_{i,k}\hookrightarrow Y^\pm$ is in fact a surjection. As $Y^\pm$ is reduced by (1), we have $\RU^\pm_{i,k}\simeq Y^\pm$.  
By Corollary \ref{coro-Ureduced},  $\RM^\pm_{i,k}$ is reduced.
\end{proof}


Proposition \ref{lem-Rikreduced} and Proposition \ref{lem-coverworst} imply Theorem \ref{intro-chart}.


\subsection{General parahoric case}
Let $I\sset [0,n]$ be non-empty. Then for each $i\in I$, the scheme $\RM^\pm_{i,k}$ is reduced by Proposition \ref{lem-Rikreduced}. 
Recall that  \begin{flalign}
	\RM^\pm_{I,k} = \bigcap_{i\in I}\rho_i\inverse(\RM^\pm_{i,k})   \label{inters}
\end{flalign}
as a schematic intersection (see \eqref{MIintersec}). Since $$\rho_i\colon LG^{\flat\circ}/L^+\sG_I^{\flat\circ}\ra LG^{\flat\circ}/L^+\sG_i^{\flat\circ}$$ is a $L^+\sG_i^{\flat\circ}/L^+\sG_I^{\flat\circ}$-bundle (which is smooth),  the scheme $\rho_i\inverse(\RM^\pm_{i,k})$ is also reduced.
As $\rho_i$ is $L^+\sG_I^{\flat\circ}$-equivariant, it follows that $\rho_i\inverse(\RM^\pm_{i,k})$ is a union of Schubert varieties in $LG^{\flat\circ}/L^+\sG^{\flat\circ}_I$. 

\begin{lemma}
	Any Schubert variety in $LG^{\flat\circ}/L^+\sG^{\flat\circ}_I$ is compatibly Frobenius-split. In particular, arbitrary intersections of unions of Schubert varieties are reduced.
\end{lemma}
\begin{proof}
	This follows from the standard theory of (affine) Schubert varieties. We refer to \cite[Theorem 4.1]{fakhruddin2025singularities} for the precise meaning of ``compatibly Frobenius-split". Since $G^{\flat\circ}$ is split and $p\neq 2$, any Schubert variety is normal by \cite[Theorem 0.3]{PR08}. It follows that any Schubert variety is compatibly Frobenius-split by \cite[Theorem 4.1]{fakhruddin2025singularities}. Then we obtain that arbitrary intersections of unions of Schubert varieties are Frobenius-split, and hence are reduced by \cite[Proposition 2.3]{gortz2001flatness}. 
\end{proof}

By the previous lemma and \eqref{inters}, we immediately obtain that $\RM^\pm_{I,k}$ is reduced. We then complete the proof of Theorem \ref{mainthm} in the general parahoric case.

\section{Application to moduli spaces of type $D$}\label{sec-application}
In this section, we apply previous results to study moduli spaces of type D in \cite{rapoport1996period}.

Let $(\bB,*)$ be a definite quaternion algebra over $\BQ$ equipped with a positive involution $*$. Let $(\bV,\pair{\ ,\ })$ be a free $\bB$-module of rank $n$, endowed with a $\BQ$-valued perfect alternating bilinear form $\pair{\ ,\ }$ which is skew-hermitian, i.e., $\pair{bv,w}=\pair{v,b^*w}$ for all $v,w\in \bV$, $b\in \bB$. Let $\bG=\bG(\bB,*,\bV,\pair{\ ,\ })$ denote the reductive group over $\BQ$ whose $R$-points for any $\BQ$-algebra $R$ are given by \begin{flalign*}
    \bG(R)\coloneqq \cbra{g\in \GL_{\bB\otimes_\BQ R}(\bV\otimes_\BQ R)\ |\ \pair{gv,gw}=c(g)\pair{v,w} \text{\ for some\ } c(g)\in R\cross }.
\end{flalign*} 
We have $\bG_{\ol{\BQ}}\simeq \GO_{2n,\ol{\BQ}}$.
Let $\bh\colon\Res_{\BC/\BR}\BG_m\ra \bG_\BR$ be a homomorphism which defines on $\bV_\BR$ a Hodge structure of type $\cbra{(-1,0),(0,-1)}$. The pairing $\pair{-,\bh(i)-}$ is a symmetric positive definite bilinear form on $\bV_\BR$. 

Then we obtain a PEL-datum $(\bB,*,\bV,\pair{\ ,\ },\bh)$ of type $D_n$. 
Let $\bE$ be the associated reflex field in the sense of \cite[1.2.5.4]{lan13}.

\begin{lem}
    We have $\bE=\BQ$.
\end{lem}
\begin{proof}
    Denote by $\bV_\BC=V^{-1,0}\oplus V^{0,-1}$ the Hodge decomposition determined by $\bh$. Note that  \begin{flalign*}
        \bE= \BQ\rbra{\Tr(b\ |\ V^{-1,0})\ |\ b\in \bB },
    \end{flalign*} 
    by \cite[Corollary 1.2.5.6]{lan13}. The characteristic polynomial of $b\in \bB$ acting on $V^{-1,0}$ is given by \begin{flalign*}
        \Char(b\ |\ V^{-1,0}) = \Char(b)^n\in \BQ[T],
    \end{flalign*}
    where $\Char(b)\in\BQ[T]$ denotes the reduced characteristic polynomial of $b$, see \cite[\S 2.3, pp. 550]{yu2021reduction}. In particular, we obtain that $\Tr(b\ |\ V^{-1,0})\in \BQ$. Therefore, $\bE=\BQ$. 
\end{proof}

\begin{remark}
    Suppose $n\geq 4$. Denote by $$\mu_\bh\coloneqq \bh_\BC(-,1)\colon \BG_{m,\BC}\ra \bG_\BC$$ the Hodge cocharacter. By the classification of Shimura data of PEL type, we have two possibilities for the (geometric) conjugacy class $\cbra{\mu_\bh}$, corresponding to the minuscule coweights $\mu_+$ and $\mu_-$ in \eqref{cochar}.
\end{remark}

The following lemma is a special case of the Morita equivalence.
\begin{lem}\label{lem-morita}
    Let $R$ be a commutative ring.
    Let $e_{11}\coloneqq \begin{psmallmatrix}
        1 &0\\ 0 &0
    \end{psmallmatrix}$ and $e_{22}\coloneqq \begin{psmallmatrix}
        0 &0\\ 0 &1
    \end{psmallmatrix}$ be two standard idempotents in the matrix algebra $M_2(R)$. Let $W$ be a (left) $M_2(R)$-module. We have $W=e_{11}W\oplus e_{22}W$ as $R$-modules. Then the functor $$W\mapsto e_{11}W$$ defines an equivalence between the category of $M_2(R)$-modules and the category of $R$-modules.
\end{lem}
\begin{proof}
    Let $W_1$ be an $R$-module. Set $W\coloneqq W_1\oplus W_1$. Viewing elements in $W$ as $2\times 1$ column vectors, the module $W$ becomes a natural $M_2(R)$-module. Clearly, the functor $W_1\mapsto W$ defines an inverse of the functor in the lemma. 
\end{proof}
\begin{defn}
    We say that the decomposition $W=e_{11}W\oplus e_{22}W$ in Lemma \ref{lem-morita} is the \dfn{Morita decomposition} of $W$. Sometimes we simply write $W=W_1\oplus W_2$.
\end{defn}

Let $p$ be an odd prime. Suppose that $B\coloneqq \bB\otimes_\BQ\BQ_p\simeq M_2(\BQ_p)$. Let $V\coloneqq \bV_{\BQ_p}=V_1\oplus V_2$ denote the Morita decomposition. Then $V_1$ is a $2n$-dimensional vector space over $\BQ_p$. Let $C\coloneqq \begin{psmallmatrix}
    0 &1\\ -1 &0
\end{psmallmatrix}$ be the Weil element in $B$ (identified with $M_2(\BQ_p)$) so that the positive involution $*$ on $B$ is given by $A^*=CA^tC\inverse$. Define \begin{flalign*}
    (x ,y)\coloneqq \pair{x, Cy}
\end{flalign*}
for $x,y\in V$. Since $C^*=-C$, the paring $(\ ,\ )$ is a non-degenerate symmetric pairing on $V$. We also use $(\ ,\ )$ to denote its restriction to the Morita component $V_1$.

\begin{lem}
    The group $\bG_{\BQ_p}$ is isomorphic to $\GO(V_1,(\ ,\ ))$.
\end{lem}
\begin{proof}
    Let $R$ be a $\BQ_p$-algebra. By Lemma \ref{lem-morita}, we have \begin{flalign}
        \GL_{B\otimes_{\BQ_p} R}(V\otimes_{\BQ_p}R) = \GL_{R}(V_1\otimes_{\BQ_p}R).  \label{51}
    \end{flalign}
    Let $x,y\in V\otimes_{\BQ_p} R$. With respect to the Morita decomposition, we can write $x=x_1+x_2$ and $y=y_1+y_2$ for $x_i,y_i\in V_i\otimes_{\BQ_p} R$ for $i=1,2$. For $g\in \GL_{B\otimes_{\BQ_p} R}(V\otimes_{\BQ_p}R)$, we write $g_1$ for the corresponding element in $\GL_R(V_1\otimes_{\BQ_p}R)$ by \eqref{51}. By construction, the condition \begin{flalign*}
        \pair{gx,gy}=c(g)\pair{x,y} \text{\ for some $c(g)\in R\cross$}
    \end{flalign*}
    is equivalent to the condition that \begin{flalign*}
        (g_1x_1,g_1y_1)=c(g_1)(x_1,y_1) \text{\ for some $c(g_1)\in R\cross$}.
    \end{flalign*}
     Hence, the lemma follows. 
\end{proof}

Assume that there exists a $\BQ_p$-basis $(e_1)_{1\leq i\leq 2n}$ of $V_1$ such that $(e_i,e_j)=\delta_{i,2n+1-i}$. Denote $\CO_B=M_2(\BZ_p)$. Let $\sL$ be a self-dual (with respect to $\pair{\ ,\ }$) $\CO_B$-lattice chain in $V$ whose first component in the Morita decomposition is given by a self-dual (with respect to $(\ ,\ )$) $\BZ_p$-lattice chain $\CL$ of $V_1$. 
Let $\sG$ denote the (affine smooth) group scheme of similitude automorphisms of $\sL$. Set \begin{equation*}
    \bK_p\coloneqq \sG(\BZ_p)\sset \bG(\BZ_p).
\end{equation*}

Applying the results of the previous sections to the case $F\coloneqq \BQ_p$, we obtain the spin local model $$\RM^\pm_\CL\ra \Spec\BZ_p.$$ It follows that $\RM^\pm_\CL$ is a topological flat $\BZ_p$-scheme, and moreover it is $\BZ_p$-flat with reduced special fiber by Theorem \ref{mainthm}. 

\begin{defn}\label{Mpmm}
    Denote by $\RM_\sL^\pm$ the functor \begin{flalign*}
        \Sch_{/\BZ_p}^\op\ra \Sets
    \end{flalign*}
    which sends a scheme $S$ over $\BZ_p$ to the set of $\CO_S$-modules $(\CF_\Lambda)_{\Lambda\in\sL}$, where $\CF_{\Lambda}$ is a locally free $\CO_S$-submodule of  $\Lambda_S\coloneqq \Lambda\otimes_{\BZ_p}\CO_S$  of rank $2n$ such that \begin{enumerate}[label=(\roman*)]
        \item for any $\Lambda\in\sL$, the $\CO_S$-submodule $\CF_\Lambda\sset \Lambda_S$ is Zariski locally direct summand;
        \item for any $\Lambda\in\sL$, $\CF_\Lambda$ is invariant under the $\CO_B$-action on $\Lambda_S$;
        \item for any inclusion $\Lambda\sset \Lambda'$ in $\sL$,  the natural map $\Lambda_S\ra \Lambda'_S$ induced by $\Lambda\hookrightarrow\Lambda'$ sends $\CF_\Lambda$ to $\CF_{\Lambda'}$; and the isomorphism $\Lambda_S\simto(p\Lambda)_S$ induced by $\Lambda\xrightarrow{p}p\Lambda$ identifies $\CF_\Lambda$ with $\CF_{p\Lambda}$;
        \item the perfect skew-Hermitian pairing $\varphi\colon \Lambda_S\times \Lambda^\#_S\ra \CO_S$ 
        induced by $\pair{\ ,\ }$ satisfies $\varphi(\CF_\Lambda,\CF_{\Lambda^\#})=0$, where $\Lambda^\#$ denotes the dual $\CO_B$-lattice of $\Lambda$ with respect to $\pair{\ ,\ }$;
        \item for any $\Lambda\in\sL$, denote by $(\CF_{\Lambda})_1$ (resp. $(\Lambda_S)_1$) the first component in the Morita decomposition of $\CF_\Lambda$ (resp. $\Lambda_S$). Then $(\CF_{\Lambda})_1$ is locally a direct summand of $(\Lambda_S)_1$ of $\CO_S$-rank $n$. We require that $(\CF_\Lambda)_1$ satisfies (LM4$\pm$) in Definition \ref{defn-spin}. 
    \end{enumerate}
\end{defn}

By Lemma \ref{lem-morita}, the map $(\CF_\Lambda)_{\Lambda\in\sL}\mapsto ((\CF_\Lambda)_1)_{\Lambda\in\sL} $ given by the Morita decomposition defines an isomorphism (cf. \cite[Lemma 10.1]{yu2021reduction}) \begin{flalign*}
    \RM^\pm_{\sL} \simto \RM^\pm_\CL.
\end{flalign*}


We recall the following definition of Rapoport and Zink; see details in \cite[\S 6.8-6.9]{rapoport1996period}.
\begin{defn}[{\cite[Definition 6.9]{rapoport1996period}}] \label{defn-naiveS}
     Let $\bK^p$ be a sufficiently small open compact subgroup in $\bG(\BA_f^p)$. Set $\bK\coloneqq \bK_p\bK^p$.
	Denote by $\CA^\pm_{\bK}$ the functor $$\CA^\pm_{\bK}\colon (\Sch/\BZ_p)^\op\lra \Sets$$ which sends any $\BZ_p$-scheme $S$ to the set of isomorphism classes of the following data.
	\begin{enumerate}[label=(\alph*)]
		\item An $\sL$-set of abelian $\CO_B$-schemes $(A_\Lambda,\iota_\Lambda)_{\Lambda\in\sL}$ in $AV(S)$ such that $A_\Lambda$ is a $2n$-dimensional abelian scheme over $S$.
		\item A $\BQ$-homogeneous polarization $\ol{\lambda}$ of the $\sL$-set $(A_\Lambda,\iota_\Lambda)_{\Lambda\in\sL}$.
		\item A $\bK^p$-level structure $$\ol{\eta}\colon H_1(A,\BA^p_f)\simeq \bV\otimes_\BQ\BA_f^p\mod \bK^p$$ that respects the bilinear forms on both sides up to a constant in $(\BA_f^p)\cross$.
		\item We further impose the following condition\footnote{cf. \cite[\S 5.3-5.4]{smithling2015moduli}.} on $(A_\Lambda,\iota_\Lambda)_{\Lambda\in\sL}$: Choose an \etale cover $T\ra S$ such that there exists an isomorphism \begin{flalign*}
        \phi\colon M((A_{\Lambda})_T)\simeq \Lambda_T
    \end{flalign*} of $\CO_B\otimes_{\BZ_p}\CO_T$-modules, where $M((A_{\Lambda})_T)$ denotes the $\CO_T$-linear dual of the first de Rham cohomology of $(A_\Lambda)_T$.  We require that the image of the covariant Hodge filtration $\phi(\Fil^1((A_\Lambda)_T))$ in $\Lambda_T$ satisfies the spin condition in the sense of condition (v) in Definition \ref{Mpmm}.
	\end{enumerate}
\end{defn}

The functor $\CA^\pm_{\bK}$ is representable by a quasi-projective scheme over $\BZ_p$, which we again denote by $\CA^\pm_{\bK}$. We sometimes say that $\CA_\bK^\pm$ is an integral moduli space of type $D$.

\begin{remark}
    \begin{enumerate}
        \item As $\bB_{\BQ_p}\simeq M_2(\BQ_p)$, the determinant condition (Kottwitz condition) in \cite[Definition 6.9]{rapoport1996period} is automatically satisfied; see \cite[Remark 1.4 (4)]{yu2021reduction}.
        \item Let $\Sh_\bK$ denote the Shimura variety (over $\BQ$) with level $\bK$ attached to the PEL-datum $(\bB,*,\bV,\pair{\ ,\ },\bh)$. Then $\Sh_{\bK,\BQ_p}$ is (isomorphic to) an open and closed subscheme of $\CA^+_{\bK,\BQ_p}$ or $\CA^-_{\bK,\BQ_p}$, depending on the choices of $\bh$. Due to the failure of the Hasse principle for $\bG$, the scheme $\CA^\pm_{\bK,\BQ_p}$ contains connected components other than $\Sh_{\bK,\BQ_p}$. See \cite[\S 2.5]{yu2021reduction} for a more detailed discussion.
    \end{enumerate}
\end{remark}

By the local model diagram (cf. \cite[Proposition 3.33]{rapoport1996period}) and the proof of \cite[Proposition 2.3]{pappas2000arithmetic}, the scheme $\CA^\pm_{\bK}$ fits into a diagram of schemes over $\BZ_p$:
\begin{equation} \label{lmd}
       \begin{gathered}
       	  \xymatrix{
		   &\wt{\CA}_{\bK}^\pm\ar[ld]_{\alpha}\ar[rd]^\beta & \\ \CA_{\bK}^\pm &  &\RM^\pm_{\sL}\simeq \RM^\pm_\CL\ ,  }
       \end{gathered}
\end{equation}
where $\alpha$ is a $\sG$-torsor, and $\beta$ is $\sG$-equivariant and smooth of relative dimension $\dim \bG$. In particular, the scheme $\CA^\pm_\bK$ is \etale locally isomorphic to $\RM^\pm_\CL$. Thus, the main results in \S \ref{spinflat} imply the following.
\begin{corollary}\label{coro-AKflat}
    The moduli space $\CA_{\bK}^\pm$ of type $D$ is a normal, Cohen--Macaulay, flat $\BZ_p$-scheme with reduced special fiber.
\end{corollary}

This proves Corollary \ref{introcoro18} in \S \ref{sec-introduction}.



\end{document}